\documentclass[10pt, final, oneside, twocolumn, letter]{IEEEtran}
\usepackage{graphicx}
\usepackage{amsmath, amsthm, mathrsfs}
\usepackage{hyperref}
\usepackage{amssymb}                            
\usepackage[dvips]{epsfig}
\usepackage{subcaption}
\usepackage[noadjust]{cite}
\usepackage{tikz}
\usepackage{caption}
\usepackage{epstopdf}
\usepackage{color}
\usepackage{multirow}

\newtheorem{theorem}{\textbf{Theorem}}
\newtheorem*{theorem*}{\textbf{Theorem}}
\newtheorem{definition}{\textbf{Definition}}
\newtheorem{prop}{\textbf{Proposition}}

\newtheorem{cor}{Corollary}
\newtheorem{lemma}{Lemma}

\title{On the Design of Attitude Observers on the Special Orthogonal Group $SO(3)$}
\author{Soulaimane Berkane and Abdelhamid Tayebi
\thanks{This work was supported by the National Sciences and Engineering Research Council of Canada (NSERC).}
\thanks{The authors are with the Department of Electrical and Computer Engineering, University of Western Ontario, London, Ontario, Canada. A. Tayebi is also with the Department of Electrical Engineering, Lakehead University, Thunder Bay, Ontario, Canada.
     {\tt\small sberkane@uwo.ca, atayebi@lakeheadu.ca} }%
}
\begin{document}
\maketitle
\begin{abstract}
We revisit the nonlinear complimentary filter on $SO(3)$, previously proposed in the literature, and provide the (time-explicit) solution to the matrix ODE governing the attitude estimation error in the absence of measurement errors. The stability and performance properties of this filter can be easily deduced from the obtained closed-from solution. Thereafter, we consider two nonlinear complimentary filters (with state-dependant gains) which are shown to exhibit improved stability and performance proprieties compared to the traditional filter. We perform robustness analysis for the three discussed attitude filters on $SO(3)$ with respect to attitude and angular velocity measurement errors. Specifically, we show that the state-dependant-gain filters may exhibit improved robustness to gyro measurement disturbances and a better disturbance attenuation levels. Simulation results are performed to confirm the obtained theoretical results.
\end{abstract}
\section{Introduction}

The ability to estimate the orientation (attitude) of a rigid body is an important feature in many engineering applications. As such, this problem has attracted the attention of many researchers and industrials for several decades. This is mainly due to the fact that there is ``no sensor" that directly measures the attitude. The attitude information is usually reconstructed using a set of body-frame measurements of known inertial vectors. Static attitude reconstruction from inertial vector measurements is one of the earliest solutions to this problem (see, for instance, \cite{Shuster1981, Markley1988}). Although simple, these methods do not perform well in the presence of measurements noise. As an alternative solution, several Kalman-type filters have been developed and successfully used in aerospace applications, although with extra care as they usually rely on linearizations and heavy computations (see, for instance, \cite{Markley2003,crassidis2007survey,crassidis2003unscented,PF2009}).
                 On the other hand, simple and yet practical linear complementary filters (for small rotational motions) have been successfully used in practical applications, \textit{e.g.,} \cite{corke2004inertial,tayebi2006attitude}, where the angular velocity is used to complement the inertial vector measurements to improve the estimation accuracy through an appropriate filtering. Nonlinear attitude filters that use the quaternion measurements have been proposed in \cite{salcudean1991globally,vik2001nonlinear,Thienel2003,bonnabel2006non,tayebi2007attitude}. More recently, nonlinear complimentary filters, evolving on $SO(3)$, have emerged and showed their ability in handling efficiently the attitude estimation problem \cite{Mahony2008,vasconcelos2008nonlinear,grip2012attitude,khosravian2012rigid,zamani2013minimum,izadi2014rigid,zlotnik2017nonlinear}. These filters have the distinctive advantage of using directly inertial vector measurements which are available on-board of most aerial and underwater vehicles; thus obviating the need of reconstructing the attitude. This class of smooth nonlinear observers guarantees, in general, \textit{almost} global asymptotic stability (AGAS), {\textit i.e.,} convergence to the actual attitude is guaranteed from any initial condition except from a set of Lebesgue measure zero. As a matter of fact, AGAS is the strongest result one aims to obtain on a compact manifold such as $SO(3)$ using time-invariant continuous control or estimation algorithms \cite{Koditschek,Bhat2000}. To overcome this topological obstruction, attitude estimators (evolving outside $SO(3)$) with global asymptotic and exponential stability properties have been proposed in \cite{batista2012sensor} and \cite{Batista2012}, respectively. The topological obstruction on $SO(3)$ has been also successfully addressed via the \textit{synergistic} hybrid technique \cite{mayhew2011hybrid,lee2015tracking,berkaneCDC2015synergistic,berkane2015construction}. Using this approach, global asymptotic hybrid attitude observers on $SO(3)$ have been proposed in \cite{lee2015observer} and global exponential hybrid attitude observers on $SO(3)$  have been proposed in \cite{berkaneACC2016observer,berkaneCDC2016observer,berkane2016design}.

Recent studies, such as \cite{izadi2014rigid,izadi2015comparison,zlotnik2017nonlinear,lee2015observer,zlotnik2016exponential}, pointed out that the nonlinear complimentary filters proposed in \cite{Mahony2008,zamani2013minimum}, which are widely used in practice, may suffer from slow convergence and robustness issues. Motivated by these recent studies, the present paper aims to conduct a rigorous performance and robustness analysis of the nonlinear complementary filter on $SO(3)$ and proposes different directions and solutions for improvement. First, we revisit the nonlinear complementary filter on $SO(3)$ proposed in \cite{Mahony2008} in the case of unbiased angular velocity measurements. We derive a closed-form (time-explicit) solution for the estimation error dynamics. The stability and performance properties of this filter can be directly deduced from the obtained solution. In particular, we derive a lower bound on the convergence time of the filter and consequently explain (rigorously) why the filter suffers from slow convergence when initialized at large attitude estimation errors. Then, we consider two state-dependent-gain nonlinear attitude estimators, evolving both on $SO(3)$, exhibiting faster convergence rates, compared to the attitude observer of \cite{Mahony2008}. The two attitude estimators share a similar structure to the observer proposed in \cite{Mahony2008} and are very similar, up to some minor details, to the filters proposed in \cite{zlotnik2017nonlinear} and \cite{zlotnik2016exponential}; which are also inspired from \cite{Thienel2003,lee2012}. The two proposed filters are, however, singularity-free compared to  \cite{zlotnik2017nonlinear} and \cite{zlotnik2016exponential}. Note that that for the sake of simplicity we ignore the integral bias adaptation law proposed in \cite{Mahony2008} which can be added in real applications without affecting the stability of the filter as shown in \cite{Mahony2008}. Furthermore, we investigate the robustness properties of these proposed nonlinear complementary filters on $SO(3)$  in the presence of bounded gyro measurement errors and small attitude measurement errors. It is shown that the newly proposed attitude filters exhibit larger robustness domains compared to the traditional constant gain filter. 
				\section{Background and preliminaries}\label{sec2}
        Throughout the paper, we use $\mathbb{R}$ and $\mathbb{R}^+$ to denote, respectively, the sets of real and nonnegative real numbers. The Euclidean norm of $x\in\mathbb{R}^n$ is defined as $\|x\|=\sqrt{x^\top x}$. For a square matrix $A\in\mathbb{R}^{n\times n}$, we denote by $\lambda_i^A, \lambda_{\mathrm{min}}^A$, and $\lambda_{\mathrm{max}}^A$ the $i$th, minimum, and maximum eigenvalue of $A$, respectively.
\\
The rigid body attitude evolves on the Special Orthogonal group defined as $SO(3) := \{ R \in \mathbb{R}^{3\times 3}|\; \mathrm{det}(R)=1,\; RR^{\top}= I \}$, where $I$ is the three-dimensional identity matrix and $R\in SO(3)$ is called a \textit{rotation matrix}. The \textit{Lie algebra} of $SO(3)$, denoted by $\mathfrak{so}(3):=\left\{\Omega\in\mathbb{R}^{3\times 3}\mid\;\Omega^{\top}=-\Omega\right\}$, is the vector space of 3-by-3 skew-symmetric matrices. Let the map $[\cdot]_\times: \mathbb{R}^3\to\mathfrak{so}(3)$ be defined such that $[x]_\times y=x\times y$, for any $x, y\in\mathbb{R}^3$, where $\times$ is the vector cross-product on $\mathbb{R}^3$. Let $\mathrm{vex}:\mathfrak{so}(3)\to\mathbb{R}^3$ denote the inverse isomorphism of the map $[\cdot]_\times$, such that $\mathrm{vex}([\omega]_\times)=\omega,$ for all $\omega\in\mathbb{R}^3$ and $[\mathrm{vex}(\Omega)]_\times=\Omega,$ for all $\Omega\in\mathfrak{so}(3)$. Defining $\mathbb{P}_a:\mathbb{R}^{3\times 3}\to\mathfrak{so}(3)$ as the projection map on the Lie algebra $\mathfrak{so}(3)$ such that $ \mathbb{P}_a(A):=(A-A^{\top})/2$, we can extend the definition of $\mathrm{vex}$ to $\mathbb{R}^{3\times 3}$ by taking the composition map $\psi := \mathrm{vex}\circ \mathbb{P}_a$ 
        such that, for a $3$-by-$3$ matrix $A:=[a_{ij}]_{i,j = 1,2,3}$, one has
        \begin{equation}\label{psi}
        \psi(A):=\mathrm{vex}\left(\mathbb{P}_a(A)\right)=\frac{1}{2}\left[\begin{array}{c}
        a_{32}-a_{23}\\a_{13}-a_{31}\\a_{21}-a_{12}
        \end{array}
        \right].
        \end{equation}
Let $|R|_I\in[0, 1]$ be the normalized Euclidean distance on $SO(3)$ which is given by
\begin{equation}\label{norm_R}
        |R|_I^2:=\frac{1}{4}\mathrm{tr}(I-R).
\end{equation}
    The attitude of a rigid body can also be represented as a rotation of angle $\theta\in\mathbb{R}$ around a unit vector axis $u\in\mathbb{S}^2$. This is commonly known as the angle-axis parametrization of $SO(3)$ and is given by the map $\mathcal{R}_{a}:\mathbb{R}\times\mathbb{S}^2\to SO(3)$ such that
            \begin{eqnarray}\label{Rod_formula}
            \mathcal{R}_{a}(\theta,u)=I+\sin(\theta)[u]_\times+(1-\cos\theta)[u]_\times^2.
            \end{eqnarray}
Alternatively, elements of $SO(3)$ can be parameterized by vectors on $\mathbb{R}^3$ through the map $\mathcal{R}_{r}:\mathbb{R}^3\to SO(3)$ such that
            \begin{align}\nonumber
            \mathcal{R}_{r}(z)&=(I+[z]_\times)(I-[z]_\times)^{-1}\\
            							&=\frac{1}{1+\|z\|^2}\left((1-\|z\|^2)I+2zz^\top+2[z]_\times\right).\label{Cayley_formula}	
            \end{align}
Equation \eqref{Cayley_formula} is often known as Cayley's formula \cite{cayley1846quelques}. Note that the matrix $(I-[z]_\times)$ is always invertible for all $z\in\mathbb{R}^3$. In fact, since $[z]_\times$ is a skew-symmetric matrix, all its eigenvalues are pure imaginary and, thus, all the eigenvalues of $I-[z]_\times$ are non-zero. The map $\mathcal{R}_r$ is a diffeomorphism between $\mathbb{R}^3$ and $SO(3)\setminus\Pi$ with $\Pi=\{R\in SO(3)\mid |R|_I=1\}$. The inverse map $\mathcal{Z}: SO(3)\setminus\Pi\to\mathbb{R}^3$ is given by
\begin{equation}\label{Rod_vec}
\mathcal{Z}(R)=\mathrm{vex}\big((R-I)(R+I)^{-1}\big)=\frac{\psi(R)}{2(1-|R|_I^2)}.
\end{equation}
The vector $\mathcal{Z}(R)\in\mathbb{R}^3$ defines the vector of Rodrigues parameters. Note that the Rodrigues vector is usually defined using unit quaternion or the angle-axis representation\cite{Shuster1993}. We prefer to use directly rotation matrices on $SO(3)$ in \eqref{Rod_vec}. It can be verified that the time derivative of the Rodrigues vector $\mathcal{Z}(R)$ along the trajectories of $\dot R=R[\xi]_\times, \xi\in\mathbb{R}^3$ is given by
\begin{equation}\label{dZ}
\frac{d}{dt}\mathcal{Z}(R)=\frac{1}{2}\big(I+[\mathcal{Z}(R)]_\times+\mathcal{Z}(R)\mathcal{Z}(R)^\top\big)\xi.
\end{equation}
It is worth pointing out that all the attitude filters derived in this paper are directly evolving on the Special Orthogonal group $SO(3)$. The introduction of the three-parameters Rodrigues vector is only for the sake of analysis. The following technical lemmas are useful throughout the paper.
\begin{lemma}\label{lemma::1}
Let $R\in SO(3)$ and $A=A^\top\in\mathbb{R}^{3\times 3}$ such that $\bar A=\frac{1}{2}(\mathrm{tr}(A)I-A)$ is positive definite. Then, the following hold
\begin{align}\label{R2_norm_theta}
\|\psi(R)\|^2=|R^2|_I^2&=4|R|_I^2(1-|R|_I^2),\\
\label{bounds::VA}
4\lambda_{\min}^{\bar{A}}|R|_I^2\leq\mathrm{tr}(A(I&-R))\leq4\lambda_{\max}^{\bar{A}}|R|_I^2,\\
\xi^2|R|_I^2(1-|R|_I^2)\leq\frac{\|\psi(AR)\|^2}{(2\lambda_{\max}^{\bar A})^2}&\leq |R|_I^2(1-\xi^2|R|_I^2)\label{psiA_ineq}
\end{align}
such that $\xi:=\lambda_{\min}^{\bar A}/\lambda_{\max}^{\bar A}$. Moreover, for all $R\in SO(3)\setminus\Pi$,
\begin{align}\label{psiA_Z}
\psi(AR)=\frac{2(I-[\mathcal{Z}(R)]_\times)}{1+\|\mathcal{Z}(R)\|^2}\bar A\mathcal{Z}(R).
\end{align}
\end{lemma}
\begin{proof}
See Appendix \ref{appendix::proof::lemma1}
\end{proof}
The following definition and characterization of Local-Input-to-State-Stability (LISS) property for nonlinear systems is needed throughout the paper and can be found in \cite{Sontag1996}.  Consider the system
\begin{equation}\label{nlsystem}
\dot x=f(x,u),
\end{equation}
where $f:\mathbb{R}^n\times\mathbb{R}^m\to\mathbb{R}^n$ is locally Lipschitz in $x$ and $u$. The input $u(t)$ is a piecewise continuous, bounded function of $t$ for all $t\geq 0$.
\begin{definition}\label{definition::ISS1}
System \ref{nlsystem} is said to be locally input-to-state stable if there exist $k_x,k_u>0,\gamma\in\mathcal{K},\beta\in\mathcal{K}\mathcal{L}$ such that
\begin{multline}
\|x(0)\|<k_x\;\mathrm{and}\;\sup_{t\geq 0}\|u(t)\|<k_u\Rightarrow\\
\|x(t)\|\leq\beta(\|x(0)\|,t)+\gamma\big(\sup_{t\geq 0}\|u(t)\|\big), \forall t\geq 0.
\end{multline}
\end{definition}
\begin{lemma}\label{lemma::ISS1}
Let $D\subset\mathbb{R}^n$ be a domain that contains the origin and $V:D\to\mathbb{R}$ be a continuously differentiable function such that
\begin{align}
\alpha_1(\|x\|)\leq V(x)\leq \alpha_2(\|x\|),\\
\dot V(x)\leq-\alpha_3(\|x\|),\forall \|x\|\geq\rho(\|u\|),
\end{align}
for all $\|x\|<r_x$ and $\|u\|<r_u$, where $\alpha_i, i=1,2,3$ and $\rho$ are class $\mathcal{K}$ functions. Then, the system \eqref{nlsystem} is locally input-to-state stable with $\gamma=\alpha_1^{-1}\circ\alpha_2\circ\rho$, $k_x=\alpha_2^{-1}\circ\alpha_1(r_x)$ and $k_u=\min\{\rho^{-1}(k_x),r_u\}$.
\end{lemma}
       \section{The attitude estimation problem}
            Let $R\in SO(3)$ denote a rotation matrix from the body fixed-frame to a given inertial reference frame. The rotation matrix $R$ evolves according to the kinematic equation
        \begin{equation}\label{kinematic}
        \dot{R}=R[\omega]_\times,
        \end{equation}
        where $\omega\in\mathbb{R}^3$ is the angular velocity expressed in the body fixed-frame. Let $\omega_y(t)$ denote the angular velocity measurement (usually provided by a gyroscope) such that
            \begin{equation}\label{omega::measured}
            \omega_y(t)=\omega(t)+n_\omega(t),
            \end{equation}
            where $n_\omega(t)$ is \textit{a priori} bounded signal that captures the gyro-bias, measurements noise and all other disturbances. Attitude information is usually extracted from body-frame measurements of know reference vectors such as those obtained from accelerometers, magnetometers or star trackers. Two non-collinear vector measurements are usually sufficient to provide an algebraic reconstruction of the attitude matrix, namely $R_y(t)$ such that 
\begin{equation}\label{R::measured}
R_y(t)=N_R(t)R(t),
\end{equation}            
where $N_R(t)\in SO(3)$ is a rotation matrix that captures all the perturbations and measurement errors that are inherent to the attitude reconstruction procedure at hand. Many (static) attitude reconstruction schemes are available, see for instance the TRIAD \cite{black1964passive}, the SVD \cite{Markley1988} and the QUEST \cite{Shuster1981}. The reconstructed attitude $R_y(t)$ is not reliable in practical applications due to measurements noise and the limited bandwidth (and sometimes the poor quality) of the inertial sensors \cite{mahony2005complementary}. \\
            The goal of the attitude complementary filter, is to fuse the available gyro measurements together with the reconstructed attitude $R_y$ (or directly the inertial vector measurements) to obtain a good (filtered) attitude estimate $\hat R$. Note that, throughout this paper, we will use the terms `filter',  `estimator' and `observer' indistinguishably.

\section{Performance Analysis for Different Nonlinear Complimentary Filters on $SO(3)$ in the Absence of Measurement Errors}
In this section, we study the stability and performance of three different deterministic nonlinear complementary attitude filters on $SO(3)$ in the case of perfect measurements, \textit{i.e.} we consider 
\begin{equation}\label{assum1}
R_y(t)=R(t)\;\textrm{and}\;\omega_y(t)=\omega(t),\;\forall t\geq 0.
\end{equation}
Although, the introduction of a filter is not necessary under the above assumption, the study of the dynamics of a given filtering scheme is conducted in the error-free case to avoid the complexities introduced by considering random noise and disturbances in the sensors measurements. In the next section, however, we will study the robustness property of the proposed filtering methods in the presence of measurement errors.

 The first discussed filter is inspired from \cite{mahony2005complementary,Mahony2008} where we, for the sake of simplicity, ignore the angular velocity bias vector. The two other filters are similar in their sructure to the filter in \cite{Mahony2008} up to a state-dependent-gain in the filter innovation term. To this purpose, we derive closed-form solutions to the differential equations governing the attitude estimation error. We believe that these results as well as the analysis methods used to conclude the stability and the explicit performance of the filters are novel. 

Consider the following nonlinear complementary attitude filter on $SO(3)$ inspired from Mahony \textit{et al.} \cite{Mahony2008}:
            \begin{align}\label{Filter_I}
			\textrm{Filter I}\;\left\{\begin{array}{rcl}
             \dot{\hat R}&=&\hat R\left[\omega_y]_\times-[\sigma\right]_\times\hat R,\\
                 			\sigma&=&-\psi(AR_y\hat R^\top),
            \end{array}\right.
            \end{align}
            where $\hat{R}\in SO(3)$ is an estimate of $R$ with $\hat R(0)=\hat R_0\in SO(3)$ and $A$ is a symmetric matrix such that $\bar A:=\frac{1}{2}(\mathrm{tr}(A)I-A)$ is positive definite. The attitude estimator \eqref{Filter_I} falls under the category of gradient-based observers on Lie groups \cite{lageman2010gradient}. In fact, for $R_y\equiv R$, the observer innovation term $\sigma=-\psi(A\tilde R)$, where $\tilde R=R\hat R^\top$ represents the attitude estimation error, can be directly obtained from the gradient of the following smooth attitude potential function on $SO(3)$
\begin{equation}\label{UA}
            U_A(\tilde R)=\mathrm{tr}(A(I-\tilde R)),
\end{equation}
which is the well-known weighted trace function on $SO(3)$ that has been widely used in the literature for the design of attitude control systems \cite{berkane2015construction,Koditschek,Sanyal2009}. The attitude filter \eqref{Filter_I} has been used in many academic and industrial applications due to its proven almost global asymptotic stability, local exponential stability and its nice filtering properties.  In the following theorem we give (for the first time) the explicit solution of the attitude estimator \eqref{Filter_I}.
\begin{theorem}\label{theorem::explicit1}
Consider the attitude kinematics system \eqref{kinematic} coupled with the attitude observer  \eqref{Filter_I} under assumption \eqref{assum1}. Then, 
\begin{itemize}
\item The closed-loop system has at least four equilibria characterized by
$
\{I\}\cup\mathcal{R}_a(\pi,\mathcal{E}(A))
$
where $\mathcal{E}(A)$ is the set of all unit eigenvectors of $A$.
\item The set of all rotations of angle $\pi$, defined by $\Pi=\{\tilde R\in SO(3)\mid |\tilde R|_I=1\}$, is invariant and non-attractive.
\item For any $\tilde R(0)\in SO(3)\setminus\Pi$, one has
\begin{equation}\label{tilde_R_explicit}
\tilde R(t)=\mathcal{R}_r\big(e^{-\bar At}\mathcal{Z}(\tilde R(0))\big),\;\forall t\geq 0.
\end{equation}
\end{itemize}
\end{theorem}
\begin{proof}
The proof of the first two items of Theorem \ref{theorem::explicit1} can be directly deduced from \cite[Proposition 1]{tayebiACC2011}. Nevertheless, for the sake of completeness, we prove these properties again. Using the product rule, and in view of \eqref{kinematic} and \eqref{Filter_I}, one obtains
      \begin{align}\nonumber
        \dot{\tilde R}&=\dot R\hat R^\top-R\hat R^\top\dot{\hat R}\hat R^\top\\
        				\nonumber	&=R[\omega]_\times\hat R^\top-R\left[\omega\right]_\times\hat R^\top+R\hat R^\top[\sigma]_\times\\\label{dtilde_R}
        					&=\tilde R[\sigma]_\times,
        \end{align}
        where the fact that $[u]_\times P^\top=P^\top[Pu]_\times$, for all $u\in\mathbb{R}^3$ and $P\in SO(3)$, has been used to obtain the last equality.
The equilibria of the closed loop-system are characterized by $\sigma=-\psi(A\tilde R)=-\mathrm{vex}(\mathbb{P}_a(A\tilde R))=0$ which by \cite[Lemma 2]{berkane2015construction} implies that $\tilde R\in \{I\}\cup\mathcal{R}_a(\pi,\mathcal{E}(A))$. Now, since $|\tilde R|_I^2=\mathrm{tr}(I-\tilde R)/4$, it follows that the time derivative of $|\tilde R|_I^2$ along the trajectories of \eqref{dtilde_R} satisfies
 \begin{align}\nonumber
\frac{d}{dt}|\tilde R|_I^2&=-\mathrm{tr}(\tilde R[\sigma]_\times)/4\\\nonumber
									 &=\mathrm{tr}(\tilde R\mathbb{P}_a(A\tilde R))/4\\\nonumber
									 &=\mathrm{tr}(\tilde R(A\tilde R-\tilde R^\top A))/8\\\label{dtilde_R_norm}
 									&=-\mathrm{tr}(A(I-\tilde R^2))/8.
\end{align}
Therefore, in view of \eqref{R2_norm_theta}-\eqref{bounds::VA}, it follows that
\begin{equation}\label{dtilde_R_ineq}
-2\lambda_{\max}^{\bar A}(1-|\tilde R|_I^2)|\tilde R|_I^2\leq\frac{d}{dt}|\tilde R|_I^2\leq-2\lambda_{\min}^{\bar A}(1-|\tilde R|_I^2)|\tilde R|_I^2,
\end{equation}
which shows that the set $\Pi$ is forward invariant and a repeller.

Now, assume that $\tilde R(0)\in SO(3)\setminus\Pi$ which implies, in view of the fact that $\Pi$ is a repeller, that $\tilde R(t)\in SO(3)\setminus\Pi$ for all future time $t\geq 0$. Therefore, the inverse map $\mathcal{Z}(\tilde R)$, defined in \eqref{Rod_vec}, exists for all $t\geq 0$ such that one has $\mathcal{R}_r(\mathcal{Z}(\tilde R(t)))=\tilde R(t)$. Making use of \eqref{dZ}, \eqref{psiA_Z} and \eqref{dtilde_R} one obtains
\begin{align}\nonumber
&\frac{d}{dt}\mathcal{Z}(\tilde R)\\\nonumber
&=\frac{1}{2}\big(I+[\mathcal{Z}(\tilde R)]_\times+\mathcal{Z}(\tilde R)\mathcal{Z}(\tilde R)^\top\big)\sigma\\\nonumber
												&=-\frac{1}{2}\big(I+[\mathcal{Z}(\tilde R)]_\times+\mathcal{Z}(\tilde R)\mathcal{Z}(\tilde R)^\top\big)\psi(A\tilde R)\\\nonumber
												&=-\big(I+[\mathcal{Z}(\tilde R)]_\times+\mathcal{Z}(\tilde R)\mathcal{Z}(\tilde R)^\top\big)\frac{(I-[\mathcal{Z}(\tilde R)]_\times)}{1+\|\mathcal{Z}(\tilde R)\|^2}\bar A\mathcal{Z}(\tilde R)\\\nonumber
												&=-\frac{1}{(1+\|\mathcal{Z}(\tilde R)\|^2)}\big(I-[\mathcal{Z}(\tilde R)]_\times^2+\mathcal{Z}(\tilde R)\mathcal{Z}(\tilde R)^\top\big)\bar A\mathcal{Z}(\tilde R)\\
												&=-\bar A\mathcal{Z}(\tilde R),\label{dtilde_Z}
\end{align}
where we have used the fact that $[\mathcal{Z}(\tilde R)]_\times^2=-\|\mathcal{Z}(\tilde R)\|^2I+\mathcal{Z}(\tilde R)\mathcal{Z}(\tilde R)^\top$ to obtain the last equality. By simple integration of \eqref{dtilde_Z}, it follows that
\begin{align}\label{Z:explicit}
\mathcal{Z}(\tilde R(t))=e^{-\bar At}\mathcal{Z}(\tilde R(0)),
\end{align}
for all $t\geq 0$, which yields \eqref{tilde_R_explicit}.
\end{proof}
Theorem \ref{theorem::explicit1} provides an explicit solution for the attitude estimator of Mahony \textit{et al.} \cite{Mahony2008}, given by equations \eqref{Filter_I}, in the absence of measurement errors. Equation \eqref{dtilde_Z} shows that the Rodrigues vector associated to the attitude estimation error follows the dynamics of a linear time-invariant system with negative definite state matrix. The three-parameters Rogriguez vector decays, therefore, exponentially fast and the explicit solution for the linear system can be derived as in \eqref{Z:explicit}. The corresponding attitude error matrix is subsequently obtained via the Cayley's formula \eqref{tilde_R_explicit}. It is worth pointing out that, although the Rodrigues vector is converging exponentially to zero, the attitude estimation error does not necessary converge exponentially fast as well. The convergence property of the norm of the attitude error is given in the following corollary.
\begin{cor}\label{corollary1}
Consider the attitude kinematics system \eqref{kinematic} coupled with the attitude observer  \eqref{Filter_I} under assumption \eqref{assum1}. Then, for any $\tilde R(0)\in SO(3)\setminus\Pi$, the Euclidean distance of the attitude error $\tilde R$ on $SO(3)$ is given by
\begin{equation}\label{corollary1::equation}
|\tilde R(t)|_I^2=\frac{\psi(\tilde R(0))^\top e^{-2\bar At}\psi(\tilde R(0))}{4(1-|\tilde R(0)|_I^2)^2+\psi(\tilde R(0))^\top e^{-2\bar At}\psi(\tilde R(0))},
\end{equation}
for all $t\geq 0$.
\end{cor}
\begin{proof}
In view of \eqref{Rod_vec}, \eqref{R2_norm_theta} it follows that 
\begin{align}\label{pf::cor1::eq1}
\|\mathcal{Z}(\tilde R(t))\|^2&=\frac{\|\psi(\tilde R(t))\|^2}{4(1-|\tilde R(t)|_I^2)}=\frac{|\tilde R(t)|_I^2}{1-|\tilde R(t)|_I^2}.
\end{align}
On the other hand, using the result of Theorem \ref{theorem::explicit1}, one has
\begin{align}\nonumber
\|\mathcal{Z}(\tilde R(t))\|^2&=\|e^{-\bar At}\mathcal{Z}(\tilde R(0))\|^2\\
										   &=\frac{\psi(\tilde R(0))^\top e^{-2\bar At}\psi(\tilde R(0))}{4(1-|\tilde R(0)|_I^2)^2},\label{pf::cor1::eq2}
\end{align}
where \eqref{Rod_vec} has been used to obtain the last equality. Simple algebraic manipulation of \eqref{pf::cor1::eq1} and \eqref{pf::cor1::eq2} yields \eqref{corollary1::equation}.
\end{proof}
Corollary \ref{corollary1} provides an explicit expression showing the evolution of the Euclidean distance $|\tilde R(t)|_I$ with respect to time. Note that it is not difficult to show that the vector $\psi(\tilde R)=\sin(\theta)u$ where $\tilde R=\mathcal{R}_a(\theta,u)$. Therefore, from \eqref{corollary1::equation}, for the same initial attitude error angle, initial attitude errors with rotation axis $u(0)$ in the direction of the larger spectrum of $\bar A$ tends to generate larger attitude errors $|\tilde R(t)|_I$ compared to an initial attitude error with rotation axis $u(0)$ in the direction of a smaller spectrum (eigenvalue) of $\bar A$. The study of the effect of the initial attitude angle on the performance of the attitude filter \eqref{Filter_I} is provided in the result of the following Corollary.
\begin{cor}\label{corollary2}
Consider the attitude kinematics system \eqref{kinematic} coupled with the attitude observer  \eqref{Filter_I} under assumption \eqref{assum1}. Then, for any $\tilde R(0)\in SO(3)\setminus\Pi$, the attitude estimation error satisfies
\begin{equation}\label{corollary2::inequality}
\underline{\beta}(|\tilde R(0)|_I,t)\leq|\tilde R(t)|_I\leq\bar\beta(|\tilde R(0)|_I,t),
\end{equation}
for all $t\geq 0$, such that $\underline{\beta}$ and $\bar\beta$ are given by
\begin{align*}
\bar{\beta}(s,t)&=\frac{se^{-\lambda_{\min}^{\bar A}t}}{\big(1-s^2(1-e^{-2\lambda_{\min}^{\bar A}t})\big)^\frac{1}{2}},\\
\underline\beta(s,t)&=\frac{se^{-\lambda_{\max}^{\bar A}t}}{\big(1-s^2(1-e^{-2\lambda_{\max}^{\bar A}t})\big)^\frac{1}{2}}.
\end{align*}
\end{cor}
\begin{proof}
First, it should be noted that the matrix $e^{-\bar At}$ is positive definite due to the fact that the matrix $-\bar At$ is symmetric. Moreover, the eigenvalues of the matrix $e^{-\bar At}$ are given by $e^{-\lambda_i^{\bar A}t},i=1,2,3,$ where $\lambda_i^{\bar A},i=1,2,3,$ are the eigenvalues of $\bar A$. Hence, in view of \eqref{R2_norm_theta}, one has
\begin{align*}
\psi(\tilde R(0))^\top e^{-2\bar At}\psi(\tilde R(0))&\leq e^{-2\lambda_{\min}^{\bar A}t}\|\psi(\tilde R(0))\|^2\\
																		    &\leq 4e^{-2\lambda_{\min}^{\bar A}t}|\tilde R(0)|_I^2(1-|\tilde R(0)|_I^2),
\end{align*}
which, in view of \eqref{corollary1::equation} and the fact that the map $x\to x/(x+a)$ is non-decreasing for all $a\geq 0$, implies that
$$
|\tilde R(t)|_I^2\leq \frac{e^{-2\lambda_{\min}^{\bar A}t}|\tilde R(0)|_I^2}{1-|\tilde R(0)|_I^2+e^{-2\lambda_{\min}^{\bar A}t}|\tilde R(0)|_I^2}=\big(\bar\beta(|\tilde R(0)|_I,t)\big)^2.
$$
Following similar steps as above, the following lower bound can be derived
$$
|\tilde R(t)|_I^2\geq \frac{e^{-2\lambda_{\max}^{\bar A}t}|\tilde R(0)|_I^2}{1-|\tilde R(0)|_I^2+e^{-2\lambda_{\max}^{\bar A}t}|\tilde R(0)|_I^2}=\big(\underline\beta(|\tilde R(0)|_I,t)\big)^2.
$$
\end{proof}
According to the upper bound on the estimation error given in Corollary \ref{corollary2}, it is clear that for small initial conditions, \textit{i.e.,} $|\tilde R(0)|_I\ll 1$, the attitude estimation error satisfies $|\tilde R(t)|_I\leq|\tilde R(0)|_I\exp(-\lambda_{\min}^{\bar A}t)$ which confirms the local exponential stability of the equilibrium point $|\tilde R|_I=0$ proved in \cite{Mahony2008}. Moreover, the convergence rate of the filter is given in the following corollary
\begin{cor}\label{corollary3}
Starting from any initial condition $\tilde R(0)\in SO(3)\setminus\Pi$, the time $t_B$ necessary to enter the ball of radius $|\tilde R(t)|_I=B$ satisfies
\begin{align}
t_B\geq\frac{1}{\lambda_{\max}^{\bar A}}\ln\left(\frac{|\tilde R(0)|_I(1-B^2)^{\frac{1}{2}}}{B(1-|\tilde R(0)|_I^2)^{\frac{1}{2}}}\right).
\end{align}
\end{cor}
\begin{proof}
Using the lower bound of \eqref{corollary2::inequality}, the time $t_B$ needs to satisfy the constraint
\begin{equation*}
\underline{\beta}(|\tilde R(0)|_I,t_B)\leq |\tilde R(t)|_I=B.
\end{equation*}
Using straightforward algebraic manipulations, the above inequality reads
\begin{align*}
e^{-\lambda_{\max}^{\bar A}t_B}\leq \frac{B^2(1-|\tilde R(0)|_I^2)}{|\tilde R(0)|_I^2(1-B^2)},
\end{align*}
which leads to the result of the corollary by taking the $\ln(\cdot)$ function on both sides of the above inequality.
\end{proof}
Therefore, according to Corollary \ref{corollary3}, it is clear that large initial estimation errors, \textit{i.e.,} $|\tilde{R}(0)|_I \to 1$, will result in low convergence rates. This fact has been numerically and experimentally observed in recent works such as \cite{lee2012,lee2015observer,zlotnik2017nonlinear}

In a tentative to improve the convergence rate of this class of attitude observers, we introduce a state-dependent scalar gain function $k: SO(3)\to\mathbb{R}^+\setminus\{0\}$, into the observer innovation term such that $\sigma_k=-k(\tilde R)\psi(A\tilde R)$ under the following assumptions:
\begin{itemize}
\item  The scalar function $k(\cdot)$ is strictly positive on $SO(3)$.
\item  The scalar function $k(\cdot)$ is a priori \textit{bounded} function on $SO(3)$.
\item  The scalar function $k(\cdot)$ is large enough for large attitude errors $\tilde R$ (for errors such that $|\tilde R|_I\to 1$).
\end{itemize}
The assumption that $k(\tilde R)$ is strictly positive is necessary to preserve the stability and convergence of the nonlinear complementary filter. In fact, following similar steps as in \eqref{dtilde_R_norm}-\eqref{dtilde_R_ineq} one can show that, with the introduction of the new innovation term $\sigma_k$, one has $d|\tilde R|_I^2/dt\leq-2\lambda_{\min}^{\bar A}k(\tilde R)(1-|\tilde R|_I^2)|\tilde R|_I^2\leq-2\lambda_{\min}^{\bar A}\underline{k}(1-|\tilde R|_I^2)|\tilde R|_I^2$ where $\underline{k}>0$ is a lower bound on $k(\tilde R)$. The assumption that $k(\tilde R)$ is bounded for all $\tilde R\in SO(3)$ is needed for practical implementation while the assumption that $k(\tilde R)$ is large for large attitude errors is introduced to increase the convergence rate for large errors. Note that the innovation term $\sigma$ of Filter I is equals $\sigma_{k_1}$ with a gain function $k_1(\tilde R)=1$ for all $\tilde R\in SO(3)$.

The nonlinear complementary filter which has been proposed recently in \cite{zlotnik2017nonlinear} can be obtain by taking the innovation term $\sigma_k$ with the gain function $k(\tilde R)=[4\lambda_{\min}^{\bar A}-U_A(\tilde R)]^{-\frac{1}{2}}$ where $U_A$ is given by \eqref{UA} for some matrix $A$. Although, the resulting filter in \cite{zlotnik2017nonlinear} guarantees faster convergence rates compared to \cite{Mahony2008} and is written directly in terms of vector measurements (no need for rotation matrix reconstruction), the implementation of \cite{zlotnik2017nonlinear} is constrained on the \textit{ellipsoid-like} set
\begin{equation}
\mathcal{S}_{\min}=\{\tilde R\in SO(3)\mid U_A(\tilde R)<4\lambda_{\min}^{\bar A}\}.
\end{equation}
In view of \eqref{bounds::VA}, the largest ball contained in the above defined ellipsoid is given by
\begin{equation}
\mathcal{B}_\xi=\{\tilde R\in SO(3)\mid |\tilde R|_I^2<\xi\},
\end{equation}
where $\xi=\lambda_{\min}^{\bar A}/\lambda_{\max}^{\bar A}$. The ball $B_\xi$ shrinks as the spread of the eigenvalues gets larger, \textit{i.e.,} as $\xi$ gets smaller. This fact limits the applicability of the attitude estimation scheme \cite{zlotnik2017nonlinear} when the eigenvalues spread $\xi$ is small; at least without considering farther modifications. Another possible choice is to consider the following gain function instead
$
k(\tilde R)=[4\lambda_{\max}^{\bar A}-U_A(\tilde R)]^{-\frac{1}{2}},
$
which is well defined on the set
\begin{equation}
\mathcal{S}_{\max}=\{\tilde R\in SO(3)\mid U_A(\tilde R)<4\lambda_{\max}^{\bar A}\},
\end{equation}
which, in view of \eqref{bounds::VA}, contains the set of all attitude errors except attitude errors of angle $\pi$, defined by the ball $\mathcal{B}_1=SO(3)\setminus\Pi$. Note that when $A$ has distinct eigenvalues, the set $\mathcal{S}_{\max}$ reduces to the whole space $SO(3)$ except only one single rotation given by $\mathcal{R}_a(\pi,v_{\max})$ with $v_{\max}$ being the unit eigenvector corresponding to the eigenvalue $\lambda_{\max}^{\bar A}$. However, small convergence rates are expected for large attitude errors in the direction of $\lambda_{\min}^{\bar A}$ when the spread of the eigenvalues is large. In fact, if one has $\tilde R=\mathcal{R}_a(\pi,v_{\min})$, with $v_{\min}$ being the unit eigenvector corresponding to the eigenvalue $\lambda_{\min}^{\bar A}$, one has $U_A(\tilde R)=4\lambda_{\min}^{\bar A}$ and therefore $k(\tilde R)=[4\lambda_{\max}^{\bar A}-4\lambda_{\min}^{\bar A}]^{-\frac{1}{2}}$ which might be small if the spread of the eigenvalues is large.

In this work, we consider the following choice for the gain function $k_2(\tilde R)=[1+\epsilon-|\tilde R|_I^2]^{-\frac{1}{2}}$ where $\epsilon>0$ is some small enough parameter. The function $k_2(\tilde R)$ satisfies all the aforementioned assumptions (positive, bounded and large for large $\tilde R$). This results in the following \textit{state-dependent-gain} nonlinear complementary filter on $SO(3)$ 
\begin{align}\label{Filter_II}
			\textrm{Filter II}\;\left\{\begin{array}{rcl}
             \dot{\hat R}&=&\hat R\left[\omega_y]_\times-[\sigma_{k_2}\right]_\times\hat R,\\
                 			\sigma_{k_2}&=&-k_2(R_y\hat R^\top)\psi(AR_y\hat R^\top),\\
            \end{array}\right.
\end{align}
where $\hat{R}\in SO(3)$ is an estimate of $R$ with $\hat R(0)=\hat R_0\in SO(3)$ and $A$ is a symmetric matrix such that $\bar A:=\frac{1}{2}(\mathrm{tr}(A)I-A)$ is positive definite.

\begin{theorem}\label{theorem::explicit2}
Consider the attitude kinematics system \eqref{kinematic} coupled with the attitude observer \eqref{Filter_II} under assumption \eqref{assum1}. Then, $\forall\epsilon,\gamma>0$ such that $\gamma<[1+\epsilon]^{-\frac{1}{2}}$, and for any $\tilde R(0)\in\mathcal{B}_{\xi_0}$ such that $\xi_0=1-\gamma^2\epsilon/(1-\gamma^2)$, one has
\begin{equation}\label{eq47}
\underline{\beta}(|\tilde R(0)|_I,t)\leq|\tilde R(t)|_I\leq\bar\beta(|\tilde R(0)|_I,t),
\end{equation}
for all $t\geq 0$, such that $\underline{\beta}$ and $\bar\beta$ are class $\mathcal{K}\mathcal{L}$ functions and are given by
\begin{align*}
\bar{\beta}(s,t)&=\frac{s}{\cosh(\gamma\lambda_{\min}^{\bar A}t)+(1-s^2)^{\frac{1}{2}}\sinh(\gamma\lambda_{\min}^{\bar A}t)},\\
\underline\beta(s,t)&=\frac{s}{\cosh(\lambda_{\max}^{\bar A}t)+(1-s^2)^{\frac{1}{2}}\sinh(\lambda_{\max}^{\bar A}t)}.
\end{align*}
\end{theorem}
\begin{proof}
Let $\epsilon,\gamma>0$ such that $\gamma\leq[1+\epsilon]^{-\frac{1}{2}}$. Then, one can verify that the scalar $\xi_0=1-\gamma^2\epsilon/(1-\gamma^2)$ is non-negative. Therefore, the ball $\mathcal{B}_{\xi_0}$ is well defined and non-empty. Let $\tilde R(0)\in\mathcal{B}_{\xi_0}$. Following similar steps as in \eqref{dtilde_R_norm}-\eqref{dtilde_R_ineq} and in view of the fact that $\sigma=-\psi(A\tilde R)/(1+\epsilon-|\tilde R|_I^2)^{\frac{1}{2}}$, one obtains
\begin{multline}\label{ineq1::proof_th2}
-2\lambda_{\max}^{\bar A}\frac{1-|\tilde R|_I^2}{(1+\epsilon-|\tilde R|_I^2)^{\frac{1}{2}}}|\tilde R|_I^2\leq\frac{d}{dt}|\tilde R|_I^2\\\leq-2\lambda_{\min}^{\bar A}\frac{1-|\tilde R|_I^2}{(1+\epsilon-|\tilde R|_I^2)^{\frac{1}{2}}}|\tilde R|_I^2.
\end{multline}
Therefore, the attitude error $|\tilde R(t)|_I^2$ is strictly decaying on $\mathcal{B}_{\xi_0}$ which implies that $\mathcal{B}_{\xi_0}$ is forward invariant and, hence, $\tilde R(t)\in\mathcal{B}_{\xi_0}$ for all $t\geq 0$. This implies that one has
\begin{align*}
0\leq|\tilde R(t)|_I^2<1-\gamma^2\epsilon/(1-\gamma^2)<1,\;\forall t\geq 0,
\end{align*}
which, after few algebraic manipulations, leads to
\begin{align*}
\gamma<\frac{(1-|\tilde R(t)|_I^2)^{\frac{1}{2}}}{(1+\epsilon-|\tilde R(t)|_I^2)^{\frac{1}{2}}}<1,\;\forall t\geq 0.
\end{align*}
It follows from \eqref{ineq1::proof_th2} that 
\begin{multline}
-2\lambda_{\max}^{\bar A}(1-|\tilde R|_I^2)^{\frac{1}{2}}|\tilde R|_I^2\leq\frac{d}{dt}|\tilde R|_I^2\\\leq-2\gamma\lambda_{\min}^{\bar A}(1-|\tilde R|_I^2)^{\frac{1}{2}}|\tilde R|_I^2.
\end{multline}
Now making use of the following integral formula
\begin{align*}
\int\frac{dx}{2x(1-x)^{\frac{1}{2}}}&=-\mathrm{arctanh}(\sqrt{1-x}):=f(x),
\end{align*}
and the comparison lemma, one obtains
\begin{equation}
-\lambda_{\max}^{\bar A}t\leq f(|\tilde R(t)|_I^2)-f(|\tilde R(0)|_I^2)\leq -\gamma\lambda_{\min}^{\bar A}t.
\end{equation}
The inverse function $f^{-1}$ is explicitly given by
\begin{align*}
f^{-1}(y)=1-\tanh^2(y)=\frac{1}{\cosh^2(y)}.
\end{align*}
Moreover, using the following identities
\begin{align*}
\cosh(a+b)&=\cosh(a)\cosh(b)+\sinh(a)\sinh(b),\\
\cosh(\mathrm{artanh}(x))&=1/\sqrt{1-x^2},\\
\sinh(\mathrm{artanh}(x))&=x/\sqrt{1-x^2},
\end{align*}
it follows that the attitude error $|\tilde R(t)|^2$ satisfies
\begin{align*}
|\tilde R(t)|_I^2&\leq f^{-1}(-\gamma\lambda_{\min}^{\bar A}t+f(|\tilde R(0)|_I^2))\\
						&\leq \frac{|\tilde R(0)|_I^2}{[\cosh(\gamma\lambda_{\min}^{\bar A}t)+(1-|\tilde R(0)|_I^2)^{\frac{1}{2}}\sinh(\gamma\lambda_{\min}^{\bar A}t)]^2}\\
						&=\big(\bar\beta(|\tilde R(0)|_I,t)\big)^2.
\end{align*}
The proof is complete.
\end{proof}
According to Theorem \ref{theorem::explicit2}, the equilibrium point $|\tilde R|_I=0$ is asymptotically stable inside the ball $\mathcal{B}_{\xi_0}$. Moreover, using the facts that $\cosh(x)=(e^{x}-e^{-x})/2$ and $e^{\frac{x}{2}}/(e^{x}-e^{-x})\leq 3^{\frac{3}{4}}/4$, it follows that
$$
\bar{\beta}(|\tilde R(0)|_I,t)\leq\frac{|\tilde R(0)|_I}{\cosh(\gamma\lambda_{\min}^{\bar A}t)}\leq \frac{3^{\frac{3}{4}}}{4}|\tilde R(0)|_Ie^{-\gamma\lambda_{\min}^{\bar A}t/2}.
$$
Hence, the convergence type of Filter II is indeed exponential inside the ball $\mathcal{B}_{\xi_0}$. Note that when $\epsilon$ is chosen sufficiently small such that $\epsilon\to 0$, one has $\xi_0\to 1$ for any $\gamma<[1+\epsilon]^{-\frac{1}{2}}$. Therefore as the parameter $\epsilon\to 0$, the region of exponential stability extends to the ball $\mathcal{B}_1$ which is equivalent to the space of all rotations less than $\pi$ angle, namely $SO(3)\setminus\Pi$. Note that the scalar $\gamma$ which appears in the exponential decay factor is \textit{independent} on the initial conditions compared to the smooth attitude estimator \eqref{Filter_I}. This results in faster convergence rates for large attitude errors. 

It is worth pointing out that the choice of the innovation term $\sigma_{k2}$ in \eqref{Filter_II} does not correspond, as far as we know, to any gradient of a potential function on $SO(3)$. In fact, this observer was designed by inspection of the dynamics of the attitude error and the desirable performance instead of the traditional systematic gradient-based method where the designer starts from a given potential function, which is typically taken as a Lyapunov candidate, and then designs the observer based on the gradient of this potential function. Our approach proposed above gives a new insight into the design of observers on $SO(3)$, in particular, and for kinematic systems on Lie groups in general.

Nevertheless, if we let $\epsilon\to 0$ and take $A=I$, it can be shown that $\sigma_{k_2}$  in \eqref{Filter_II} is related to the gradient of the non-differentiable potential function
$
\Phi(\tilde R)=1-[1-|\tilde R|_I^2]^{\frac{1}{2}},
$
inspired from the solution to the optimal kinematic problem on $SO(3)$ \cite{Saccon2010} and the work by \cite{lee2012} on attitude tracking. We have introduced an arbitrary weighting matrix $A$ as an additional tuning parameter and a small scalar $\epsilon$ that allows to remove the singularity at $180^{\circ}$ while preserving the advantage of faster convergence rates obtained when using the gradient of the non-differentiable potential function $\Phi$.

Another interesting choice for the state-dependent-gain function $k(\cdot)$ is the more aggressive function $k_3(\tilde R)=[1+\epsilon-|\tilde R|_I^2]^{-1}$ for some small enough $\epsilon>0$. The function $k_3(\tilde R)$ satisfies the needed assumptions (positive, bounded and large for large $\tilde R$). This results in the following version of the nonlinear complementary filter on $SO(3)$ 
\begin{align}\label{Filter_III}
			\textrm{Filter III}\;\left\{\begin{array}{rcl}
             \dot{\hat R}&=&\hat R\left[\omega_y]_\times-[\sigma\right]_\times\hat R,\\
                 			\sigma_{k_3}&=&-k_3(R_y\hat R^\top)\psi(AR_y\hat R^\top),
            \end{array}\right.
\end{align}
where $\hat{R}\in SO(3)$ is an estimate of $R$ with $\hat R(0)=\hat R_0\in SO(3)$ and $A$ is a symmetric matrix such that $\bar A:=\frac{1}{2}(\mathrm{tr}(A)I-A)$ is positive definite. Note that the attitude estimation scheme proposed in \cite{zlotnik2016exponential} can be recovered (in the bias-free case) from the above attitude estimator by taking $A=I$ and setting $\epsilon\to 0$. This innovation term (with $A=I$ and $\epsilon\to 0$) can be obtained from the gradient of the following \textit{barrier}-like potential function
$
\Psi(\tilde R)=-\ln(1-|\tilde R|_I^2).
$
However, the innovation term $\sigma_{k_3}$ in \eqref{Filter_III} is not, as far as we know, a consequence of a gradient of any potential function but a novel design choice that has been introduced to obtain the desirable performance demonstrated in the following theorem.
\begin{theorem}\label{theorem::explicit3}
Consider the attitude kinematics system \eqref{kinematic} coupled with the attitude observer \eqref{Filter_III} under assumption \eqref{assum1}. Then, $\forall\epsilon,\gamma>0$ such that $\gamma<[1+\epsilon]^{-1}$, and for any $\tilde R(0)\in\mathcal{B}_{\xi_0}$ such that $\xi_0=1-\gamma\epsilon/(1-\gamma)$, one has
\begin{equation}
\underline{\beta}(|\tilde R(0)|_I,t)\leq|\tilde R(t)|_I\leq\bar\beta(|\tilde R(0)|_I,t),
\end{equation}
for all $t\geq 0$, such that
\begin{align*}
\bar{\beta}(|\tilde R(0)|_I,t)&=|\tilde R(0)|_Ie^{-\gamma\lambda_{\min}^{\bar A}t},\\
\underline\beta(|\tilde R(0)|_I,t)&=|\tilde R(0)|_Ie^{-\lambda_{\max}^{\bar A}t}.
\end{align*}
\end{theorem}
\begin{proof}
Let $\epsilon,\gamma>0$ such that $\gamma\leq[1+\epsilon]^{-1}$. Then, one can verify that the scalar $\xi_0=1-\gamma\epsilon/(1-\gamma)$ is non-negative. Therefore, the ball $\mathcal{B}_{\xi_0}$ is well defined and non-empty. Let $\tilde R(0)\in\mathcal{B}_{\xi_0}$. Following similar steps as in \eqref{dtilde_R_norm}-\eqref{dtilde_R_ineq} and in view of the fact that $\sigma_{k_3}=-\psi(A\tilde R)/(1+\epsilon-|\tilde R|_I^2)$, one obtains
\begin{multline}\label{ineq1::proof_th2}
-2\lambda_{\max}^{\bar A}\frac{1-|\tilde R|_I^2}{1+\epsilon-|\tilde R|_I^2}|\tilde R|_I^2\leq\frac{d}{dt}|\tilde R|_I^2\\\leq-2\lambda_{\min}^{\bar A}\frac{1-|\tilde R|_I^2}{1+\epsilon-|\tilde R|_I^2}|\tilde R|_I^2.
\end{multline}
Therefore, the attitude error $|\tilde R(t)|_I^2$ is strictly decaying on $\mathcal{B}_{\xi_0}$ which implies that $\mathcal{B}_{\xi_0}$ is forward invariant and, hence, $\tilde R(t)\in\mathcal{B}_{\xi_0}$ for all $t\geq 0$. This implies that one has
\begin{align*}
0\leq|\tilde R(t)|_I^2<1-\gamma\epsilon/(1-\gamma)<1,\;\forall t\geq 0,
\end{align*}
which, after few algebraic manipulations, leads to
\begin{align*}
\gamma<\frac{1-|\tilde R(t)|_I^2}{1+\epsilon-|\tilde R(t)|_I^2}<1,\;\forall t\geq 0.
\end{align*}
It follows from \eqref{ineq1::proof_th2} that 
\begin{align*}
-2\lambda_{\max}^{\bar A}|\tilde R|_I^2\leq\frac{d}{dt}|\tilde R|_I^2\leq-2\gamma\lambda_{\min}^{\bar A}|\tilde R|_I^2.
\end{align*}
which yields the result of the theorem using the comparison lemma.
\end{proof}
It should be mentioned that the three discussed filters above (Filter I, Filter II and Filter III) all share the same  performance properties for small attitude errors (local performance). This is due to the fact that, for $\epsilon$ sufficiently small, one has the term $[1+\epsilon-|\tilde R|_I^2]\to 1$ for small values of $|\tilde R|_I^2$. Consequently the innovation terms for the three filters become identical and hence the performance (convergence, filtering...etc). The difference between the three filters is remarkable as the attitude error increases. To illustrate this, we plot the variations of the norm of the innovation terms $\sigma_{k_1}, \sigma_{k_2}$ and $\sigma_{k_3}$ for all the three different filters. Let $\tilde R=\mathcal{R}_a(\theta,e_3)$ where $\theta\in[0,2\pi]$ and $e_3=[0,0,1]^\top$. Consider a weighting matrix $A=\mathrm{diag}([1,2,3])$ and let us choose different values for $\epsilon=0.1, 0.01$ and $0.001$.
\begin{figure}[t!]
    \centering
    \begin{subfigure}[t]{0.5\textwidth}
        \centering
        \includegraphics[scale=0.3]{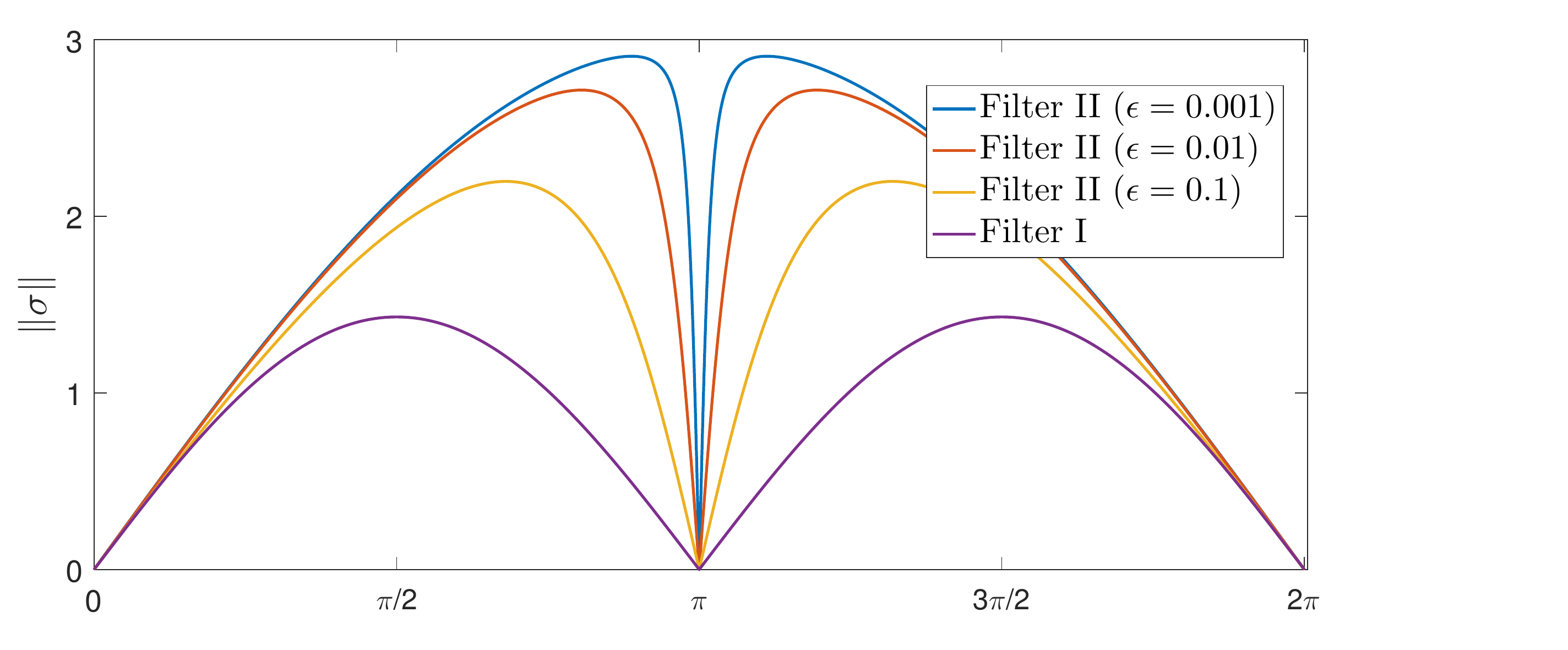}
        \caption{}
        \label{figure::sigma2}
    \end{subfigure}%
    \\
    \begin{subfigure}[t]{0.5\textwidth}
        \centering
        \includegraphics[scale=0.3]{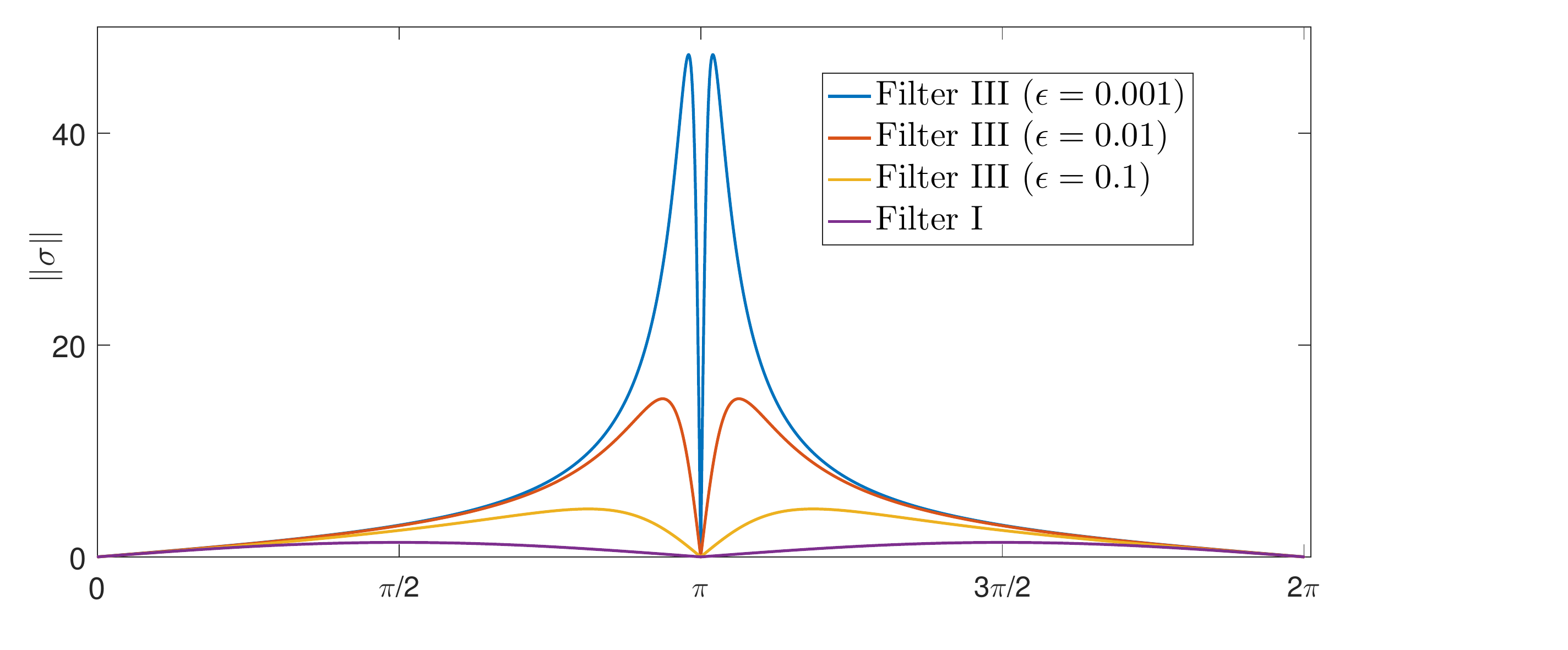}
        \caption{}
        \label{figure::sigma3}
    \end{subfigure}
    \caption{Variations of the norm of the innovation term $\sigma$ with respect to the attitude error angle $\theta$: (a) Filter II (b) Filter III.}
\end{figure}
%
It can be seen from Fig. \ref{figure::sigma2} and Fig. \ref{figure::sigma3} that the innovation term $\sigma$ for all the three different filters is bounded (for a fixed $\epsilon$) for all attitude angles. Both Filter II and Filter III use a larger (in terms of norm) correction term when the attitude error is large compared to the traditional estimation scheme given by Filter I. As $\epsilon$ is chosen smaller, the correction term becomes larger for large attitude errors. Moreover, Filter III employs a more aggressive correction term compare to Filter II. This explains the larger picks in Fig. \ref{figure::sigma3}. On the other hand, for all filters the $\sigma$ term vanishes at attitudes of angle $180^{\circ}$ around the eigen-axis $e_3$ which represents one of the undesired equilibria for the filters characterized by $\mathcal{R}_a(\pi,\mathcal{E}(A))=\cup_{i=1,2,3}\mathcal{R}_a(\pi,e_i)$.
\section{Robustness Analysis for Different Nonlinear Complementary Filters on $SO(3)$ in the Presence of Measurement Errors}
In this section we aim to study the robustness of the nonlinear complimentary filters proposed in the previous section to the following measurement errors:
\begin{itemize}
\item Bounded errors in the angular velocity measurements such as noise, bias, disturbances...etc.
\item Small errors in the attitude measurements.
\end{itemize}
As far as we know, robustness on the compact manifold $SO(3)$ has not been formulated before, at least in the context we study here. In fact, it is not clear how to define the meaning of \textit{divergence} and \textit{instability} which are necessary to justify the notion of ``robustness". In dynamical systems theory, a system is said to be unstable if at least one state variable in the system evolves without bounds (unbounded). The meaning of \textit{unbounded} state is obviously related to the chosen \textit{metric} (distance) on the given configuration space. For Euclidean spaces, for example, a state that evolves unbounded means that its Euclidean norm grows to infinity ($\infty$); which represents the maximum distance that the metric allows.

However,  the states of $SO(3)$ (rotation matrices) are naturally bounded with respect to any chosen smooth metric thanks to the geometry of the manifold. For our present purpose it is justified to relax the notion of instability on $SO(3)$ as follows: 
\begin{definition}
Given a Riemanian metric on $SO(3)$, a dynamical system on $SO(3)$ is said to be unstable if the state variable of the system (namely the attitude matrix $\tilde R\in SO(3)$) evolves to the manifold of maximum distance.
\end{definition}
According to Definition 3, let us choose the canonical Riemnian metric on $SO(3)$, called also the Euclidean metric. This metric results in the Euclidean distance defined in \eqref{norm_R}. For a given rotation matrix $\tilde R\in SO(3)$, the maximum distance $|\tilde R|_I=1$ is obtained when $\tilde R\in \Pi$; manifold of all rotations of angle $180^{\mathrm{o}}$. Interestingly, the Rodrigues vector $\mathcal{Z}(\tilde R)\in\mathbb{R}^3$ grows unbounded when $\tilde R\in \Pi$. This motivates to study the robustness of the dynamics of the Rodrigues vector $\mathcal{Z}(\tilde R)$ with respect to measurement errors in the traditional sense of robustness on Euclidean spaces. More specifically, existing results on Input-to-State-Stability (ISS) on Euclidean spaces can be directly applied to the dynamics of the Rodrigues vector $\mathcal{Z}(\tilde R)$ to derive conclusions about the robustness of the proposed filters' error dynamics.
\subsection{Robustness Study to Gyro Measurement Errors}
Here we assume that the gyro measurements are given according \eqref{omega::measured} for some bounded error vector $n_\omega$. We also consider perfect attitude information such that $R_y\equiv R$. Following similar steps as in \eqref{dtilde_R} and \eqref{dtilde_Z}, it can be verified that the dynamics of the Rodrigues vector (for the three different discussed filters) in the presence of angular velocity measurement errors are given by the following differential equations:
{\small
\begin{align}
\label{dZ1}
\textrm{Filter I: }\dot{\mathcal{Z}}(\tilde R)=-k_1(\mathcal{Z}(\tilde R))\bar A\mathcal{Z}(\tilde R)-g(\mathcal{Z}(\tilde R))\hat Rn_\omega,\\
\label{dZ2}
\textrm{Filter II: }\dot{\mathcal{Z}}(\tilde R)=-k_2(\mathcal{Z}(\tilde R))\bar A\mathcal{Z}(\tilde R)-g(\mathcal{Z}(\tilde R))\hat Rn_\omega,\\
\label{dZ3}
\textrm{Filter III: }\dot{\mathcal{Z}}(\tilde R)=-k_3(\mathcal{Z}(\tilde R))\bar A\mathcal{Z}(\tilde R)-g(\mathcal{Z}(\tilde R))\hat Rn_\omega,
\end{align}
}where the scalar valued functions $k_1,k_2$ and $k_3$ are given by the following expressions 
{\small\begin{align*}
k_1(\mathcal{Z}(\tilde R))&=1,\\
k_2(\mathcal{Z}(\tilde R))&=\left(\frac{1+\|\mathcal{Z}(\tilde R)\|^2}{1+\epsilon(1+\|\mathcal{Z}(\tilde R)\|^2)}\right)^{\frac{1}{2}},\\
k_3(\mathcal{Z}(\tilde R))&=\frac{1+\|\mathcal{Z}(\tilde R)\|^2}{1+\epsilon(1+\|\mathcal{Z}(\tilde R)\|^2)},
\end{align*}}
and $g(\mathcal{Z}(\tilde R))=\frac{1}{2}\big(I+[\mathcal{Z}(\tilde R)]_\times+\mathcal{Z}(\tilde R)\mathcal{Z}(\tilde R)^\top\big)$. Our goal in this subsection is to study the ISS property of the above dynamical systems with respect to bounded gyro disturbances $n_\omega$. Before doing so important  remarks are in order.

It can be noticed from \eqref{dZ1}-\eqref{dZ3} that, for small attitude estimation errors, the transfer function from the disturbance signal $n_\omega$ to the attitude vector $\mathcal{Z}(\tilde R)$ satisfies
$$
H(s)=\frac{1}{2}(sI+\bar A)^{-1},
$$
Note that all the proposed filters have the same transfer function $H(s)$ as above near the desired equilibrium point. The cutoff frequencies of $H(s)$ are given by
$$
2\pi f_{c_i}=\lambda_i^{\bar{A}},
$$
where $\lambda_i^{\bar{A}}$ denotes the $i-$th eigenvalue of $\bar A$. Therefore, as the magnitude of the eigenvalues of $\bar A$ decreases the cutoff frequency of the filters decreases and high frequency components of the disturbance signal are filtered. On the other hand, the eigenvalues of $\bar A$ directly affect the speed of convergence of the filters as well. Larger values of $\lambda_i^{\bar A}$ means faster convergence rates. This shows the trade-off that one would expect when considering the traditional nonlinear complimentary filter \eqref{Filter_I} (constant gain). The attitude filters proposed in \eqref{Filter_II} and \eqref{Filter_III}, however, solve this conflict between the speed of convergence and the cutoff frequency by considering a state-dependent time-varying gain which increases  as the attitude estimation error grows and decreases to $1$ for small attitude estimation errors. Therefore, one may pick up small values of $\lambda_i^{\bar A}$ to guarantee better filtering while still maintaining a good convergence speed when using the attitude filters proposed in this paper.

Now let us study rigorously the ISS property of the proposed attitude filtering schemes. First, let us prove the following interesting ``existence" result that motivates the need to carefully investigate the robustness of the nonlinear complementary filter on $SO(3)$ discussed in the previous section.
\begin{prop}\label{proposition::robustness1}
Consider the dynamics \eqref{dZ1} obtained from the error dynamics of Filter I. Assume that there exists $i\in\{1,2,3\}$ such that
\begin{align*}
\hat R(t)n_\omega(t)&=-2\lambda_{i}^{\bar A}\mathcal{Z}(\tilde R(0))(2\lambda_{i}^{\bar A}t+1)^{-\frac{1}{2}},\\
\bar A\mathcal{Z}(\tilde R(0))&=\lambda_i^{\bar A}\mathcal{Z}(\tilde R(0)),\quad \|\mathcal{Z}(\tilde R(0))\|=1.
\end{align*}
Then  $\mathcal{Z}(\tilde R(t))=\mathcal{Z}(\tilde R(0))(2\lambda_{i}^{\bar A}t+1)^{\frac{1}{2}}$ for all $t\geq 0$.
\end{prop}
\begin{proof}
First we show that, under the conditions of Proposition \ref{proposition::robustness1}, the direction of the Rodrigues vector $\mathcal{Z}(\tilde R(t))$ remains constant for all times.
In view of \eqref{dZ1}, one has
\begin{multline}\label{dtilde_Z3}
\frac{d}{dt}\|\mathcal{Z}(\tilde R)\|^2=-2\mathcal{Z}(\tilde R)^\top\bar A\mathcal{Z}(\tilde R)-\\(1+\|\mathcal{Z}(\tilde R)\|^2)\mathcal{Z}(\tilde R)^\top\hat Rn_\omega.
\end{multline}
Therefore, using \eqref{dZ1} and \eqref{dtilde_Z3}, it can be shown that $\zeta:=\frac{\mathcal{Z}(\tilde R)}{\|\mathcal{Z}(\tilde R)\|}$ satisfies the following differential equation
\begin{equation}\label{dtilde_Z4}
\frac{d}{dt}\zeta=[\zeta]_\times^2\bar A\zeta-\frac{1}{2}\Big([\zeta]_\times-\frac{[\zeta]_\times^2}{\|\mathcal{Z}(\tilde R)\|}\Big)\hat Rn_\omega.
\end{equation}
Hence, it is straightforward to see that, under the assumption of the proposition, one has
$$
\frac{d}{dt}\zeta(0)=0,
$$
which implies that the initial direction is an equilibrium point of the non-autonomous dynamics \eqref{dtilde_Z4} and, hence, the direction of the Rodrigues vector $\mathcal{Z}(\tilde R(t))$ is constant for all $t\geq 0$. Therefore, the angle between the two signals $\mathcal{Z}(\tilde R(t))$ and $\hat R(t)n_\omega(t)$ remains zero for all times. It follows, from equation \eqref{dtilde_Z3} that
\begin{multline*}
\frac{d}{dt}\|\mathcal{Z}(\tilde R)\|^2=-2\lambda_i^{\bar A}\|\mathcal{Z}(\tilde R)\|^2+\\2\lambda_{i}^{\bar A}(1+\|\mathcal{Z}(\tilde R)\|^2)\|\mathcal{Z}(\tilde R)\|(2\lambda_{i}^{\bar A}t+1)^{-\frac{1}{2}}.
\end{multline*}
It can be checked, by direct substitution in the above equation, that $\|\mathcal{Z}(\tilde R(t))\|=(2\lambda_i^{\bar A}t+1)^{\frac{1}{2}}$ is a solution. This proves the proposition.
\end{proof}
In Proposition \ref{proposition::robustness1}, it is shown that the traditional gradient-based nonlinear attitude observer on $SO(3)$ is not ISS with respect to bounded angular velocity measurements disturbances. In fact, one can construct a bounded and vanishing disturbance that prevents the observer from converging to the actual attitude. If we consider a particular time-dependent vanishing measurement disturbance $\hat Rn_\omega(t)$ with an initial attitude error $|\tilde R(0)|=1/\sqrt{2}$, or equivalently $\|\mathcal{Z}(\tilde R(0))\|=1$, which corresponds to an angle of rotation equals $90^{\mathrm{o}}$, the attitude error tends to the undesired manifold, \textit{i.e.,} $|\tilde R|_I \to 1$ ($\|\mathcal{Z}(\tilde R)\|\to\infty$) where the error angle is maximum and equals $180^{\mathrm{o}}$. It should be mentioned that the result of Proposition \eqref{proposition::robustness1} does not intend to question the applicability of the nonlinear complimentary filter of \cite{Mahony2008}. Nevertheless, the discussions of this section have motivated us to think more rigorously about the issue of robustness and convergence speed when designing attitude observers on $SO(3)$.
\begin{theorem}\label{theorem::ISS}
System \eqref{dZ1} (resp. \eqref{dZ2} and \eqref{dZ3}) is locally input-to-state stable. In particular, for all $r>0$ and $0<\varrho<1$, there exists $\beta_1$ (resp. $\beta_2$ and $\beta_3$) $\in\mathcal{K}\mathcal{L}$ such that for all $\|\mathcal{Z}(\tilde R(0))\|<r$ and $\sup_{t\geq 0}\|n_\omega(t)\|<k_{u1}$ (resp. $k_{u2}$ and $k_{u3}$), one has
\begin{align*}
\|\mathcal{Z}(\tilde R(t))\|\leq\beta_i(\|\mathcal{Z}(\tilde R(0))\|,t)+\gamma_i(\sup_{t\geq 0}\|n_\omega(t)\|),
\end{align*}
for $i=1,2,3$ with $k_{ui}=\gamma_i^{-1}(r)$, $\gamma_i(s)=\varsigma_i(r)s/2\varrho\lambda_{\min}^{\bar A}$ and 
\begin{align*}
\varsigma_1(r)&=(1+r^2),\\
\varsigma_2(r)&=((1+r^2)(1+\epsilon+\epsilon r^2))^{\frac{1}{2}},\\
\varsigma_3(r)&=1+\epsilon+\epsilon r^2.
\end{align*}
\end{theorem}
Before addressing the proof of Theorem \ref{theorem::ISS} some remarks are in order. Theorem \ref{theorem::ISS} shows that all the attitude filters (Filter I, Filter II and Filter III) discussed in this paper are Locally Input-to-State Stable (LISS) in the sense of Definition \ref{definition::ISS1}. It should be mentioned that the conclusion of LISS can be directly inferred from the fact that the unforced ($n_\omega\equiv 0$) systems \eqref{dZ1}-\eqref{dZ3} are globally asymptotically stable by using the result of \cite[Lemma I.1]{Sontag1996}. However, the conclusions of Theorem \ref{theorem::ISS} give ``explicitly" the bounds $k_{ui}, i=1,2,3$ on the disturbance $n_\omega$ where LISS holds for all the proposed filters.  The result of Theorem \ref{theorem::ISS} allows us to compare rigorously the robustness of these filters to gyro measurement errors.

For a given constant $r>0$ in Theorem \ref{theorem::ISS}, the explicit bounds on $\sup_{t\geq 0}\|n_\omega(t)\|$ can be derived as
\begin{align*}
k_{u1}&=\frac{\varrho\lambda_{\min}^{\bar A}r}{1+r^2},\\
k_{u2}&=\frac{\varrho\lambda_{\min}^{\bar A}r}{((1+r^2)(1+\epsilon+\epsilon r^2))^{\frac{1}{2}}},\\
k_{u3}&=\frac{\varrho\lambda_{\min}^{\bar A}r}{1+\epsilon+\epsilon r^2}.
\end{align*}
The constant $r>0$, in Theorem \ref{theorem::ISS}, can be arbitrarily large to cover all initial conditions for the attitude error $\tilde R\in SO(3)$. For Filter I, as $r$ gets larger the value of $k_{u1}$, which corresponds to the bound on the allowed disturbances, gets smaller. This fact suggests that as we start closer to large attitude errors the robustness to small measurement gyro disturbances may be lost. In contrast, for Filter III for example, for any large $r>$ we can always choose the parameter $\epsilon$ small enough to such that  $k_{u3}$ is also large and therefore the allowed bound on the gyro disturbances are much larger. To put this together, by setting $\epsilon$ small enough such that $\epsilon<r^2/(1+r^2)$, it can be verified that 
$$
k_{u1}<k_{u2}<k_{u3}.
$$
Consequently, it can be concluded that Filter III has the best robustness to gyro measurement errors while Filter I exhibits a reduced robustness compared to the other two proposed filters. Moreover, by letting $\epsilon\to 0$ and $r\to\infty$, it can be noticed that $k_{u1}\to 0, k_{u2}\to\rho\lambda_{\min}^{\bar A}$ and $k_{u3}\to\infty$ which (in this case) leads to conclude that Filter I is not ISS, Filter II is ISS with respect to all disturbances such that $\sup_{t\geq 0}\|n_\omega(t)\|<\varrho\lambda_{\min}^{\bar A}$ and Filter III has the Global ISS property.

\begin{proof}[Proof of Theorem \ref{theorem::ISS}]
Consider the following ISS-Lyapunov functions candidate
$$
V_i(\mathcal{Z}(\tilde R))=\frac{1}{2}\|\mathcal{Z}(\tilde R(t))\|^2,\;i=1,2,3.
$$
The time derivative of $V_1$ (resp. $V_2$ and $V_3$) along the trajectories of \eqref{dZ1} (resp. \eqref{dZ2} and \eqref{dZ3}) satisfies
\begin{align*}
\dot{V}_i(\mathcal{Z}(\tilde R))&=-k_i(\mathcal{Z}(\tilde R))\mathcal{Z}(\tilde R)^\top\bar A\mathcal{Z}(\tilde R)-\\
												&\hspace{0.5cm}\frac{1}{2}(1+\|\mathcal{Z}(\tilde R)\|^2)\mathcal{Z}(\tilde R)^\top\hat Rn_\omega\\
										    &\leq-\lambda_{\min}^{\bar A}(1-\varrho)k_i(\mathcal{Z}(\tilde R))\|\mathcal{Z}(\tilde R)\|^2+\\
										    &\hspace{0.5cm}\frac{1}{2}\|\mathcal{Z}(\tilde R)\|(1+\|\mathcal{Z}(\tilde R)\|^2)\\
										    &\hspace{0.5cm}\left(\|n_\omega\|-\frac{2k_i(\mathcal{Z}(\tilde R))\varrho\lambda_{\min}^{\bar A}\|\mathcal{Z}(\tilde R)\|}{1+\|\mathcal{Z}(\tilde R)\|^2}\right),
\end{align*}
for $i=1,2,3$. Assume that $\|\mathcal{Z}(\tilde R)\|\leq r$. Then, it is clear that for all $\|\mathcal{Z}(\tilde R)\|\geq\rho(\|n_\omega\|)$, with $\rho(s)=\gamma(s)=(1+r^2)s/2k_i(r)\varrho\lambda_{\min}^{\bar A}$, one has
\begin{align*}
\dot{V}_i(\mathcal{Z}(\tilde R))\leq -\lambda_{\min}^{\bar A}(1-\varrho)k_i(\mathcal{Z}(\tilde R))\|\mathcal{Z}(\tilde R)\|^2.
\end{align*}
Applying the result of Lemma \ref{lemma::ISS1} concludes the proof.
\end{proof}
Another interesting feature of the attitude estimation scheme given by Filter III is demonstrated in the following theorem when the tuning scalar $\epsilon\to 0$.
\begin{theorem}\label{theorem::optimal}
Consider the attitude kinematics system \eqref{kinematic} coupled with the attitude observer  \eqref{Filter_III}. Assume that $\tilde R(0)\in SO(3)\setminus\Pi$. Then the innovation term $\sigma$ in \eqref{Filter_III}, with $A=aI>0$ and $\epsilon=0$, minimizes the following cost functional
\begin{multline}\label{J}
J(\sigma)=\sup_{n_\omega\in\mathcal{N}}\Big\{\lim_{t\to+\infty}\Big[2\ln(1+\|\mathcal{Z}(\tilde R(t))\|^2)+\\\int_0^t\big((2a-\frac{1}{\gamma^2})\mathcal{Z}(\tilde R)^\top\mathcal{Z}(\tilde R)+\frac{1}{2a}\sigma^\top\sigma-\gamma^2n_\omega^\top n_\omega\big)d\tau\Big]\Big\},
\end{multline}
where $\gamma^2>\frac{1}{2a}$ and $\mathcal{N}$ is the set of locally bounded disturbances, with a value function $J^*=2\ln(1+\|\mathcal{Z}(\tilde R(0))\|^2)$. Moreover, the achieved disturbance attenuation level is
\begin{multline}
\big(4a-\frac{1}{\gamma^2}\big)\int_0^{\infty}\mathcal{Z}(\tilde R(t))^\top\mathcal{Z}(\tilde R(t))dt\leq\\\gamma^2\int_0^{\infty}n_\omega(t)^\top n_\omega(t)dt+2\ln(1+\|\mathcal{Z}(\tilde R(0))\|^2).
\end{multline}
\end{theorem}
\begin{proof}
Recall from \eqref{dZ} and \eqref{dtilde_R} that the dynamics of the Rodrigues vector are written as
\begin{align}\nonumber
\frac{d}{dt}\mathcal{Z}(\tilde R)&=\frac{1}{2}\big(I+[\mathcal{Z}(\tilde R)]_\times+\mathcal{Z}(\tilde R)\mathcal{Z}(\tilde R)^\top\big)\big(\sigma-\hat Rn_\omega\big)\\
											   &:=g(\mathcal{Z}(\tilde R))\big(\sigma-\hat Rn_\omega\big).\nonumber
\end{align}
Consider the following Lyapunov function candidate
\begin{equation}
V(\mathcal{Z}(\tilde R))=\frac{1}{2}\ln(1+\|\mathcal{Z}(\tilde R)\|^2).
\end{equation}
The Lie Derivative of $V$ along $g$ satisfies
\begin{align*}
L_gV(\mathcal{Z}(\tilde R))&=\nabla V(\mathcal{Z}(\tilde R))g(\mathcal{Z}(\tilde R))\\
										  &=\frac{\mathcal{Z}(\tilde R)^\top}{1+\|\mathcal{Z}(\tilde R)\|^2}g(\mathcal{Z}(\tilde R))\\
										  &=\frac{1}{2}\mathcal{Z}(\tilde R)^\top.
\end{align*}
Let $W_1=\gamma^2I$ and consider the following auxiliary system
\begin{align}\label{auxiliary}
\frac{d}{dt}\mathcal{Z}(\tilde R)=W_1^{-1}g(\mathcal{Z}(\tilde R))\big(L_gV(\mathcal{Z}(\tilde R))\big)^\top+g(\mathcal{Z}(\tilde R))\sigma.
\end{align}
with
$
\sigma=\frac{1}{2}\alpha(\mathcal{Z}(\tilde R)):=-W_2^{-1}\big(L_gV(\mathcal{Z}(\tilde R))\big)^\top=-a\mathcal{Z}(\tilde R),
$
where $W_2=\frac{1}{2a}I$. Then, the auxiliary system \eqref{auxiliary} becomes
\begin{align}
\frac{d}{dt}\mathcal{Z}(\tilde R)=-\frac{2a-\frac{1}{\gamma^2}}{4}\big(1+\|\mathcal{Z}(\tilde R)\|^2\big)\mathcal{ Z}(\tilde R),
\end{align}
which is clearly globally asymptotically stable as long as the scalar $2a-\frac{1}{\gamma^2}$ is strictly positive.
Consequently, using the result of \cite[Theorem 5.1]{krstic1998inverse}, it follows that
$$
\sigma=\alpha(\mathcal{Z}(\tilde R))=-2a\mathcal{Z}(\tilde R)=-a\frac{\psi(\tilde R)}{1-|\tilde R|_I^2}
$$
solves the inverse optimal $\mathcal{H}_\infty$ problem by minimizing the cost functional
\begin{multline}
J(\sigma)=\sup_{n_\omega\in\mathcal{N}}\Big\{\lim_{t\to+\infty}\Big[2\ln(1+\|\mathcal{Z}(\tilde R(t))\|^2)+\\\int_0^t\big(l(\mathcal{Z}(\tilde R))+\sigma^\top W_2\sigma-n_\omega^\top W_1n_\omega\big)d\tau\Big]\Big\},
\end{multline}
where
\begin{align*}
&l(x)\\
&=-4\big(L_gV(x)W_1^{-1}L_gV(x)^\top-L_gV(x)W_2^{-1}L_gV(x)^\top\big)\\
&=\big(2a-\frac{1}{\gamma^2}\big)x^\top x.
\end{align*}
Substituting $\sigma=\alpha(\mathcal{Z}(\tilde R))$ in \eqref{J} and using the fact that $J(\sigma)\leq J^*$ it follows that
\begin{multline}
\int_0^\infty\big((4a-\frac{1}{\gamma^2})\mathcal{Z}(\tilde R)^\top\mathcal{Z}(\tilde R)-\gamma^2n_\omega^\top n_\omega\big)dt\leq \\J^*=2\ln(1+\|\mathcal{Z}(\tilde R(0))\|^2),
\end{multline}
which proves the result.
\end{proof}
Theorem \ref{theorem::optimal} shows that the choice of the observer innovation term $\sigma$ in \eqref{Filter_III} in the ideal case where $\epsilon\to 0$ solves a meaningful inverse optimal $\mathcal{H}_\infty$ optimization problem. Moreover, a bound on the disturbance attenuation level is obtained. This result leads naturally to conclude on the ISS-type robustness of the attitude estimation scheme \eqref{Filter_III}, when $A=aI$ and $\epsilon=0$. 
\subsection{Robustness Study to Attitude Errors}
In this subsection, we assume that the attitude information $R_y$ is obtained according to \eqref{R::measured} for some \textit{small} perturbation attitude matrix $N_R\in SO(3)$. We also consider perfect gyro measurements such that $\omega_y\equiv \omega$. This allows us to study the two robustness problems separately. 

The new ``available" attitude error is given by $\tilde R_y=R_y\hat R^\top=N_RR\hat R=N_R\tilde R$. The contaminated attitude error $\tilde R_y$ will be used in the innovation term $\sigma$ in \eqref{Filter_I}, \eqref{Filter_II} and \eqref{Filter_III} for the three different versions of the nonlinear complimentary filter. For simplicity of discussions, we consider in this work only the case where $N_R$ is a perturbation rotation of small angle in the direction of the rotation $\tilde R$. Explicitly, if the orientation $\tilde R$ is described by $\mathcal{R}_a(\theta,u)$ for some $\theta\in\mathbb{R}$ and $u\in\mathbb{S}^2$ then we consider $N_R=\mathcal{R}_a(n_\theta,u)$ for some small $n_\theta\ll 1$. This implies that the available attitude error satisfies $\tilde R_y=\mathcal{R}_a((\theta+n_\theta),u)$. Using the fact that $\mathcal{Z}(\mathcal{R}_a(x,v))=\tan(x/2)v$ for all $x\in\mathbb{R}$ and $v\in\mathbb{S}^2$, it can be verified that 
\begin{align*}
\mathcal{Z}(\tilde R_y)&=\frac{\tan(\theta/2)+\tan(n_\theta/2)}{1-\tan(\theta/2)\tan(n_\theta/2)}u\\
									&\simeq\frac{\tan(\theta/2)+n_\theta/2}{1-\tan(\theta/2)n_\theta/2}u\\
									&=\mathcal{Z}(\tilde R)+\frac{n_\theta(1+\|\mathcal{Z}(\tilde R)\|^2)}{2-n_\theta\|\mathcal{Z}(\tilde R)\|}u,
\end{align*}
where we have used the following first order approximations $\tan(x)\simeq x$ for all small enough $x\in\mathbb{R}$. Now, we need to re-evaluate the expression of the innovation term $\sigma$ in terms of $\mathcal{Z}(\tilde R_y)$. In view of \eqref{psiA_Z}, the expression of $\sigma$ for the three filters is given by
$$
\sigma=-2k_i(\mathcal{Z}(\tilde R_y))\frac{(I-[\mathcal{Z}(\tilde R_y)]_\times)}{1+\|\mathcal{Z}(\tilde R_y)\|^2}\bar A\mathcal{Z}(\tilde R_y),\; i=1,2,3.
$$
Moreover, one has
$$
1+\|\mathcal{Z}(\tilde R_y)\|^2=\frac{(n_\theta^2+4)(1+\|\mathcal{Z}(\tilde R)\|^2)}{(2-n_\theta\|\mathcal{Z}(\tilde R)\|)^2}\simeq \frac{4(1+\|\mathcal{Z}(\tilde R)\|^2)}{(2-n_\theta\|\mathcal{Z}(\tilde R)\|)^2},
$$
where the second order term in $n_\theta^2$ was neglected ($n_\theta\ll 1$).
On the other hand, recall that the dynamics of $\mathcal{Z}(\tilde R)$ satisfies $\dot{\mathcal{Z}}(\tilde R)=g(\mathcal{Z}(\tilde R))\sigma$ which implies that
\begin{align*}
&\frac{1}{2}\frac{d}{dt}\|\mathcal{Z}(\tilde R)\|^2\\
&=\mathcal{Z}(\tilde R)^\top g(\mathcal{Z}(\tilde R))\sigma\\
																		&=\frac{1}{2}(1+\|\mathcal{Z}(\tilde R)\|^2)\mathcal{Z}(\tilde R)^\top\sigma\\
																		&=-\frac{k_i(\mathcal{Z}(\tilde R_y))}{4}\mathcal{Z}(\tilde R)^\top\bar A\mathcal{Z}(\tilde R_y)(2-n_\theta\|\mathcal{Z}(\tilde R)\|)^2\\
																		&\lesssim-k_i(\mathcal{Z}(\tilde R_y))\lambda_{\min}^{\bar A}\|\mathcal{Z}(\tilde R)\|^2-\\
																		&\qquad\frac{1}{2}k_i(\mathcal{Z}(\tilde R_y))\lambda_{\min}^{\bar A}n_\theta\|\mathcal{Z}(\tilde R)\|(1-\|\mathcal{Z}(\tilde R)\|^2),
\end{align*}
where again higher order terms in $n_\theta$ were neglected. Therefore, one concludes that 
$$
\frac{1}{2}\frac{d}{dt}\|\mathcal{Z}(\tilde R)\|^2\leq-\lambda_{\min}^{\bar A}(1-\varrho)k_i(\mathcal{Z}(\tilde R_y))\|\mathcal{Z}(\tilde R)\|^2,
$$
for all $\|\mathcal{Z}(\tilde R)\|\geq\rho(|n_\theta|)$ such that $\|\mathcal{Z}(\tilde R(0))\|<r$ and $\rho(s)=(1+r^2)s/2$. Therefore, according to Lemma \ref{lemma::ISS1}, the dynamics of $\mathcal{Z}(\tilde R)$ (for all three filters) are LISS for all $\|\mathcal{Z}(\tilde R(0))\|<r$ and $\sup_{t\geq 0}|n_\theta(t)|<2r/(1+r^2)$. As it is noticed, the upper bound on the allowed attitude measurements $n_\theta(t)$ decreases as the initial condition gets larger. This results can be intuitively explained by the fact that, for large attitude errors close enough to $180^{\circ}$, the attitude noise can \textit{mislead} the innovation term $\psi(A\tilde R)$ to change the direction of the correction and therefore a correction is applied in the wrong direction which causes the attitude to get closer to $180^{\circ}$. If we are unlucky enough, small noise can cause chattering near the undesired manifold of all rotations of angle $180^{\circ}$. The reader is referred to some recent works on hybrid observers on $SO(3)$ where hysteresis-like switching mechanisms have been proposed to guarantee global stability results with robustness to small measurements noise \cite{lee2015observer,berkaneACC2016observer,berkaneCDC2016observer}.
\section{Vector measurements formulation of the proposed coplementary filters on $SO(3)$}\label{section::almost}
The nonlinear complementary filters discussed in the previous sections were written in terms of the attitude information $R_y(t)$ which is not available, in practice, directly using any sensor. However, body-frame measurements of constant known inertial vectors can be obtained using different sensors such as accelerometers, magnetometers, star trackers, cameras...etc. We assume we have $n\geq 2$ vector measurements
$$
b_i=R^\top r_i,\; i=1,\cdots,n,
$$
where $r_i$ are some known constant inertial vectors. Moreover, we assume that at least two vector measurements $b_i$ are noncollinear. This is a standard assumption in attitude estimation which is necessary to recover the full attitude information from the available data. To implement one of the discussed attitude filters on $SO(3)$ in practice, we need to reconstruct the  attitude matrix $R_y$ using some static attitude determination algorithms such that $R_y=f_{\textrm{reconst}}((b_i,r_i)_{1\leq i\leq n})$. Obviously if the measurements $b_i$ are perfect then the reconstruction gives perfect attitude such that $R_y\equiv R$. However, this is not realistic as the noise in the vector measurements $b_i$ is probably to propagate to $R_y$. Attitude reconstruction schemes are likely to be senstitive to noise which motivates \cite{Mahony2008} to explicitly formulate the traditional nonlinear complementary filter using directly available measurements $b_i$.

To do so, we use the results from \cite[Proposition 5]{berkane2015construction} to derive the following identities
\begin{align}\label{explicit::psiA}
\psi(AR\hat R^\top)=\frac{1}{2}\hat R\sum_{i=1}^n\rho_i(b_i\times\hat R^\top r_i),
\end{align}
where $A=\sum_{i=1}^n\rho_ir_ir_i^\top$ such that $\rho_i, i=1,\cdots n$ are positive scalars. Note that under the assumption that two vectors $b_1$ and $b_2$ are noncollinear, one guarantees that the positive semidefinite matrix $A$ has rank greater or equal $2$. Therefore, the matrix $\bar A=\frac{1}{2}(\mathrm{tr}(A)-A)$ can be shown to have full rank (positive definite) which allows to use it in \eqref{Filter_I}, \eqref{Filter_II} and \eqref{Filter_III}. It remains to express the norm $|\tilde R|_I^2$ which appears in \eqref{Filter_II} and \eqref{Filter_III} in the expression of the state-dependent gains.

Let $b_1$ and $b_2$ be two (non-collinear) body-frame vector measurements corresponding to the inertial unit vectors $r_1$ and $r_2$ such that $b_1=R^\top r_1$ and $b_2=R^\top r_2$. Let us define the vectors $u_1=r_1/\|r_1\|, u_2=(r_1\times r_2)/\|r_1\times r_2\|$ and $u_3=u_1\times u_2$ along with their corresponding body-frame vectors $w_1=b_1/\|b_1\|, w_2=(b_1\times b_2)/\|b_1\times b_2\|$ and $w_3=w_1\times w_2$. Then, one can verify that
\begin{align}\label{explicit::normR}
|\tilde R|_I^2&=\frac{1}{8}\sum_{i=1}^3\|w_i-\hat R^\top u_i\|^2,
\end{align}
which is a quite convenient formula for the computation of $k_2(\tilde R)$ and $k_3(\tilde R)$ in \eqref{Filter_II} and \eqref{Filter_III}. Consequently, the results in \eqref{explicit::psiA} and \eqref{explicit::normR} allow to write the proposed attitude filters in \eqref{Filter_I}, \eqref{Filter_II} and \eqref{Filter_III} explicitly in terms of vector measurements without the need to reconstruct the attitude matrix $R_y$.
\section{Implementation Aspects and Numerical results}
This section presents numerical examples and comparisons among the nonlinear complimentary attitude filters discussed in this paper. First, we derive the discrete-version of the nonlinear complementary filter on $SO(3)$ for practical implementation purposes. The class of nonlinear complementary filters on $SO(3)$ discussed in this paper can be written, in the continuous setting, as 
\begin{align}\label{dhatR_integration}
\dot{\hat R}(t)=\hat R(t)[\hat\omega(t)]_\times,\quad\hat R(0)\in SO(3),
\end{align}
where estimated angular velocity $\hat\omega$ is given by
$$
\hat\omega(t)=\omega_y(t)-\hat R^\top(t)\sigma(t)
$$ and $\sigma(t)=-k(R_y(t)\hat R^\top(t))\psi(AR_y(t)\hat R^\top(t))$ such that $k(\cdot)$ depends on the type of filter used (Filter I, Filter II and Filter III). Assume that during the time interval $[t_k,t_{k+1})$, where $k\in\mathbb{N}$ and $t_0=0$, the estimated angular velocity $\hat\omega(t)$. This is a realistic assumption for small integration step sizes. Consequently in view of \eqref{dhatR_integration} it follows that
$$
\frac{d}{dt}\left(\hat Re^{[\hat\omega(t_k)t]_\times}\right)=0,\quad t\in[t_k,t_{k+1}).
$$
Exact integration of the above equation between $t_k$ and $t_{k+1}$ yields the following update step on $SO(3)$
\begin{align}\label{dhatR_discrete}
\hat R(t_{k+1})=\hat R(t_k)e^{[\hat\omega(t_k)(t_{k+1}-t_k)]_\times},\quad k\in\mathbb{N}.
\end{align}
Note that the exponential map on $SO(3)$ has a compact formula for quick computation (instead of using high order Taylor series) given by the map $\mathcal{R}_a$ in \eqref{Rod_formula} such that $e^{[x]_\times}=\mathcal{R}_a(\|x\|,x/\|x\|)$ for all $x\in\mathbb{R}^3$. Moreover, it is worth pointing out that the discrete integration rule \eqref{dhatR_discrete} can be lifted to the quaternion space (using the quaternion multiplication rule) to simplify the computations. The resulting integration scheme can be verified to be equivalent to the discrete quaternion integration proposed in \cite{Hua2014}.

Consider the kinematics of the attitude system \eqref{kinematic} with the following angular velocity input signal
$$
\omega(t)=\begin{bmatrix}
\sin(0.3t)\\
0.7\sin(0.2t+\pi)\\
0.5\sin(0.1t+\pi/3)
\end{bmatrix}(\mathrm{rad/s}),
$$
and initial condition $R(0)=I$. We implement a simulation of the real kinematic system \eqref{kinematic} through the integration scheme on $SO(3)$ above using a high sampling rate of $1000$ Hz. We assume that the gyro measurements of the angular velocity are obtained at $200$ Hz and are contaminated by a white noise with zero mean and standard standard deviation equals $0.1(\mathrm{rad/s})$, see Fig. \ref{wy}. 
\begin{figure}[h!]
\centering
\includegraphics[scale=0.35]{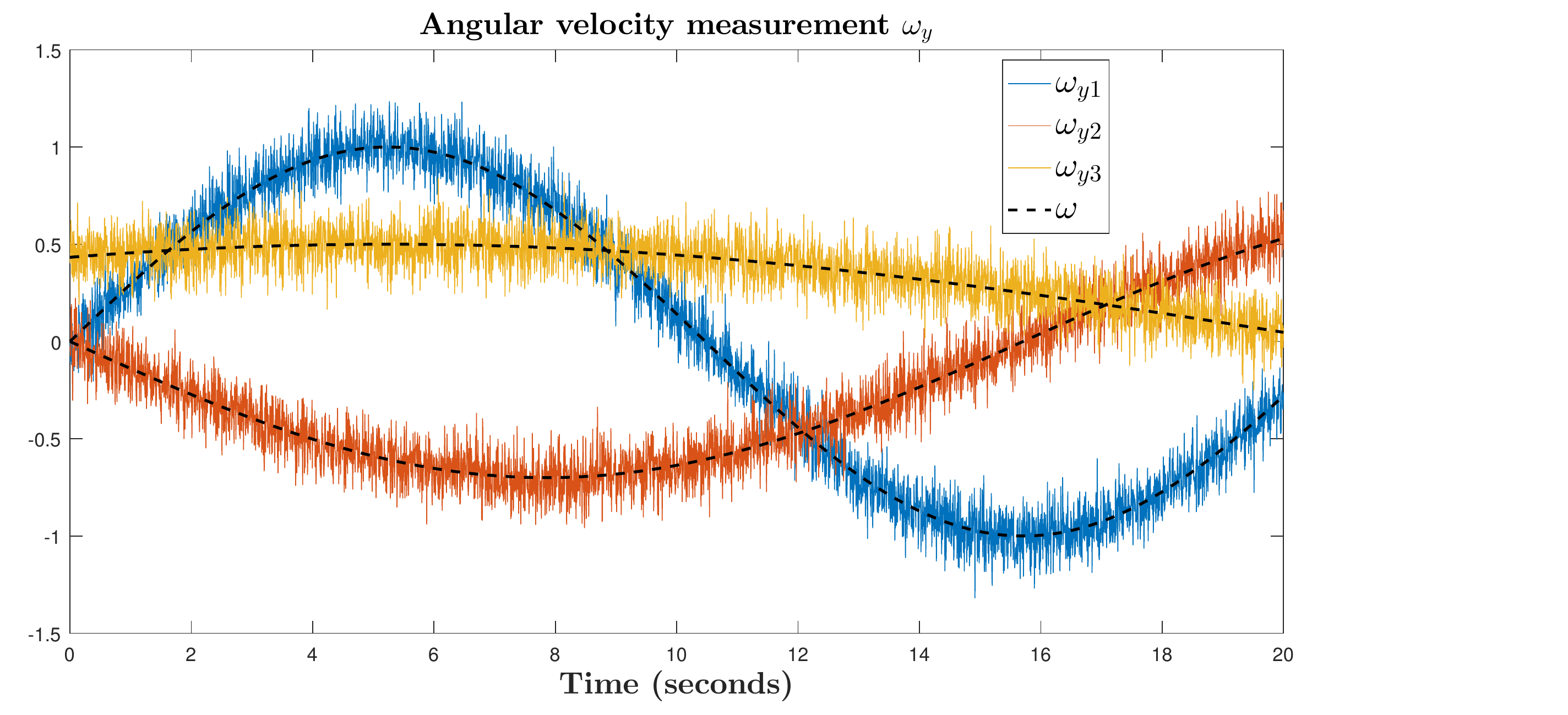}
\caption{}
\label{wy}
\end{figure}
We also consider body-frame measurements $b_1$ and $b_2$ (same sampling frequency of $200$ Hz) of two non-collinear inertial vectors given by $r_1=[1,-1,1]^{\top}/\sqrt{3}$ and $r_2=[0,0,1]^{\top}$. We also consider additional white noise with zero mean and standard deviation equals 0.1 for both vector measurements $b_1$ and $b_2$.
\begin{figure}[h!]
\centering
\includegraphics[scale=0.35]{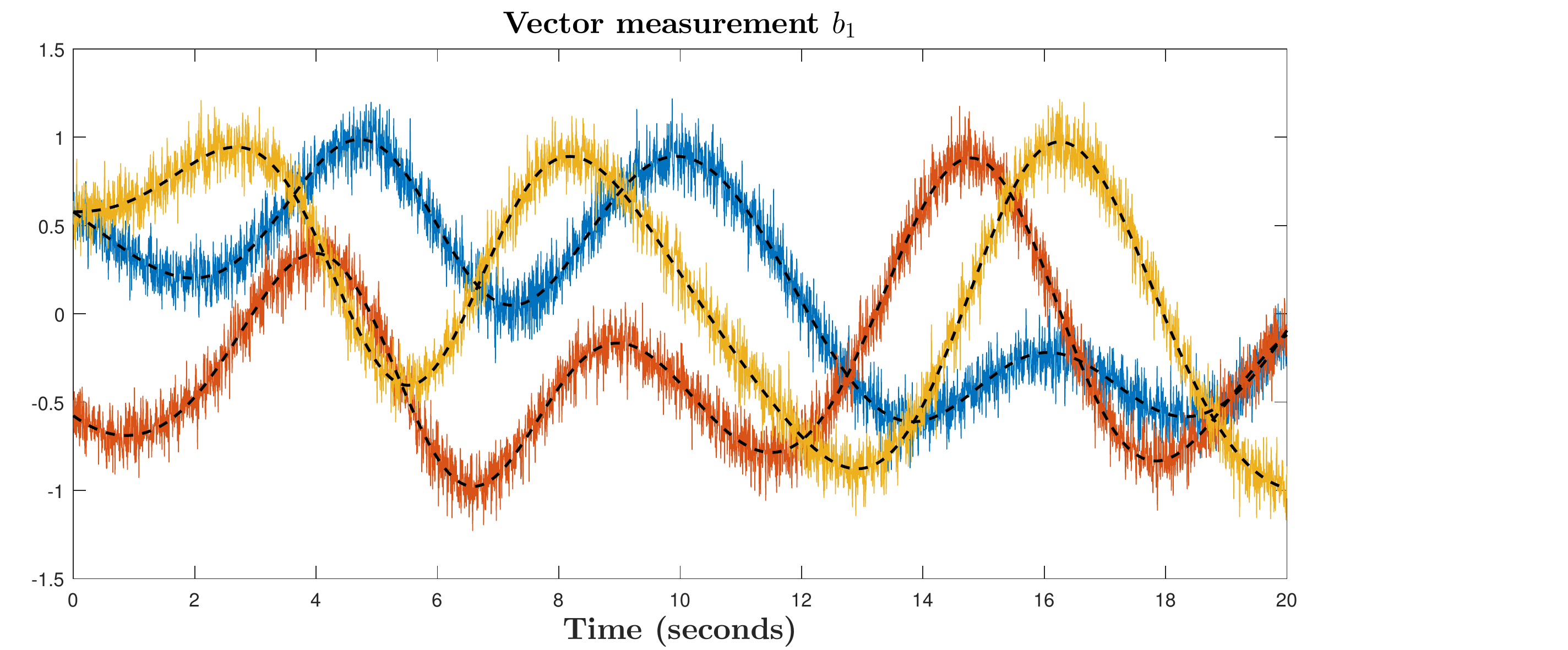}
\caption{}
\label{b1}
\end{figure}
\begin{figure}[h!]
\centering
\includegraphics[scale=0.35]{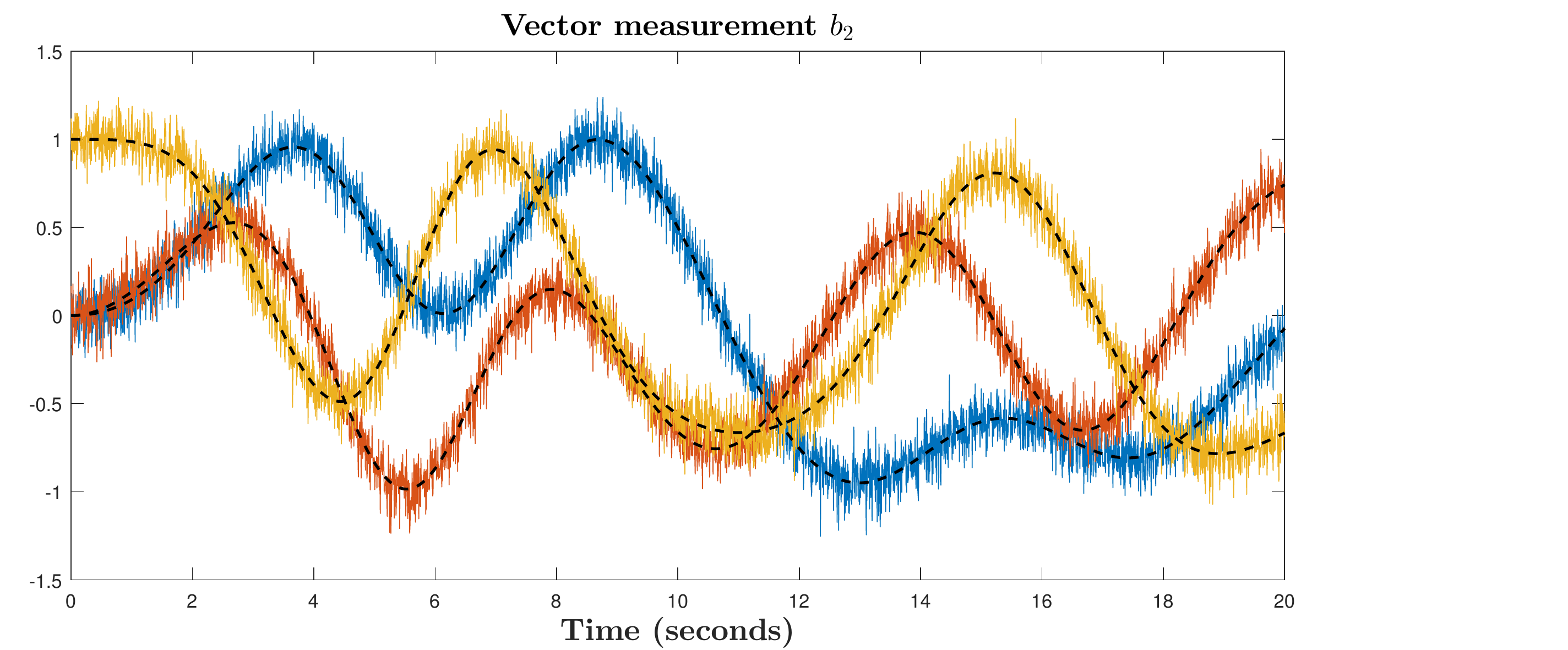}
\caption{}
\label{b2}
\end{figure}
All attitude errors are initialized at an attitude $\hat R(0)=\mathcal{R}_a(\pi-10^{-1},[1,0,0]^\top)$. The two vectors $r_1$ and $r_2$ are weighted with gains $\rho_1=1$ and $\rho_2=2$, respectively. Therefore, the corresponding weighting matrix is given by
 $$
 A=\rho_1r_1r_1^\top+\rho_2r_2r_2^\top=\frac{1}{3}\begin{bmatrix}
 1 &-1 &1\\
 -1&1&-1\\
 1&-1&7
 \end{bmatrix}.
 $$
The formulas \eqref{explicit::psiA}-\eqref{explicit::normR} are used to explicitly express the innovation term $\sigma$ for the three different filters (Filter I, Filter II and Filter III). The parameter $\epsilon$ used in Filter II and Filter III innovation term is chosen small and equals $\epsilon=10^{-2}$. The updated attitude estimates $\hat R$ are obtained at a frequency of $200$ Hz which corresponds to the frequency of the measurements.
%
%
%
\begin{figure}[h!]
\centering
\includegraphics[scale=0.35]{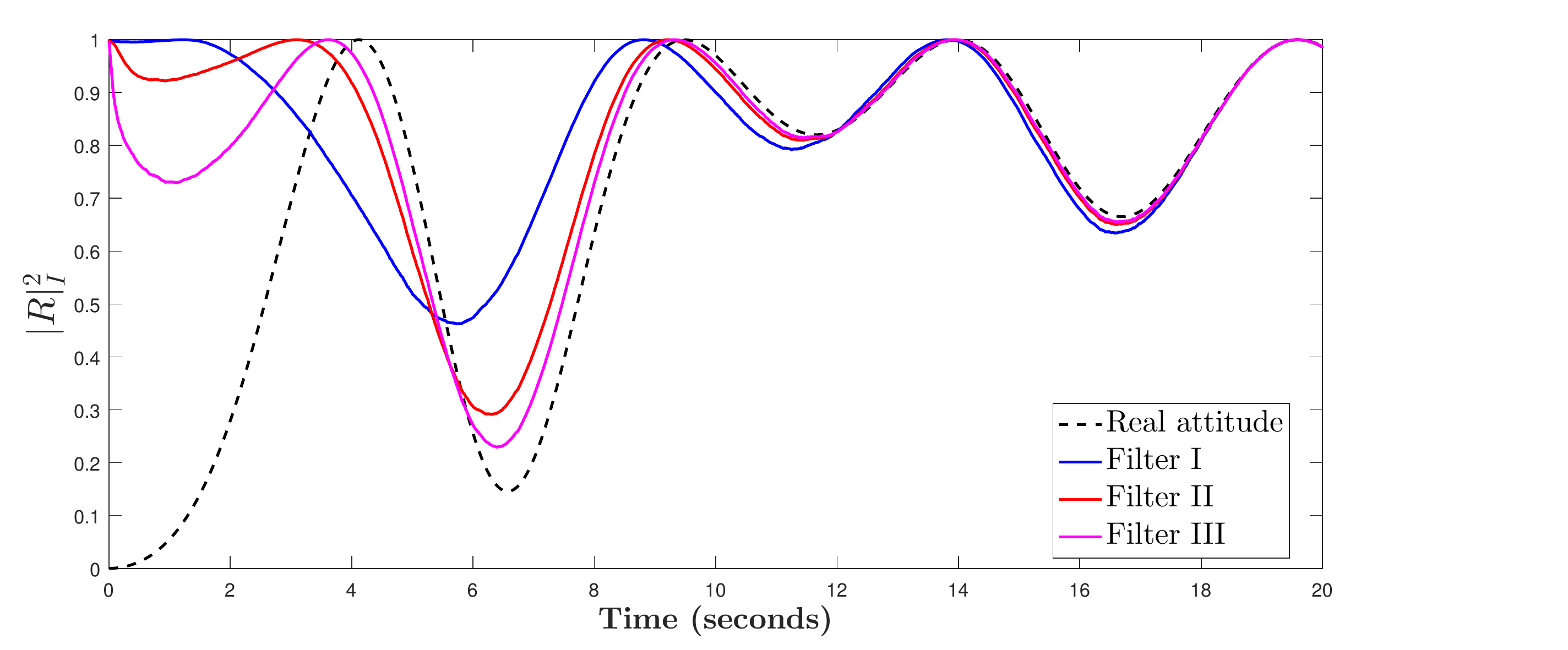}
\caption{Attitude estimates norm versus time for Filter I, Filter II and Filter III.}
\label{figure::normR}
\end{figure}
Attitude estimates norms for the three discussed filters are plotted in Fig. \ref{figure::normR}. As expected, Filter II and Filter III behave better than the constant gain filter (Filter I) int terms of convergence rate. Especially Filter III is able to correct its attitude in a faster time compared to the two other filters. Note that the three filters have an identical behaviour near the origin of attitude error so no performance is lost (locally) when introducing the state dependent gain filters.
\begin{figure}[h!]
\centering
\includegraphics[scale=0.3]{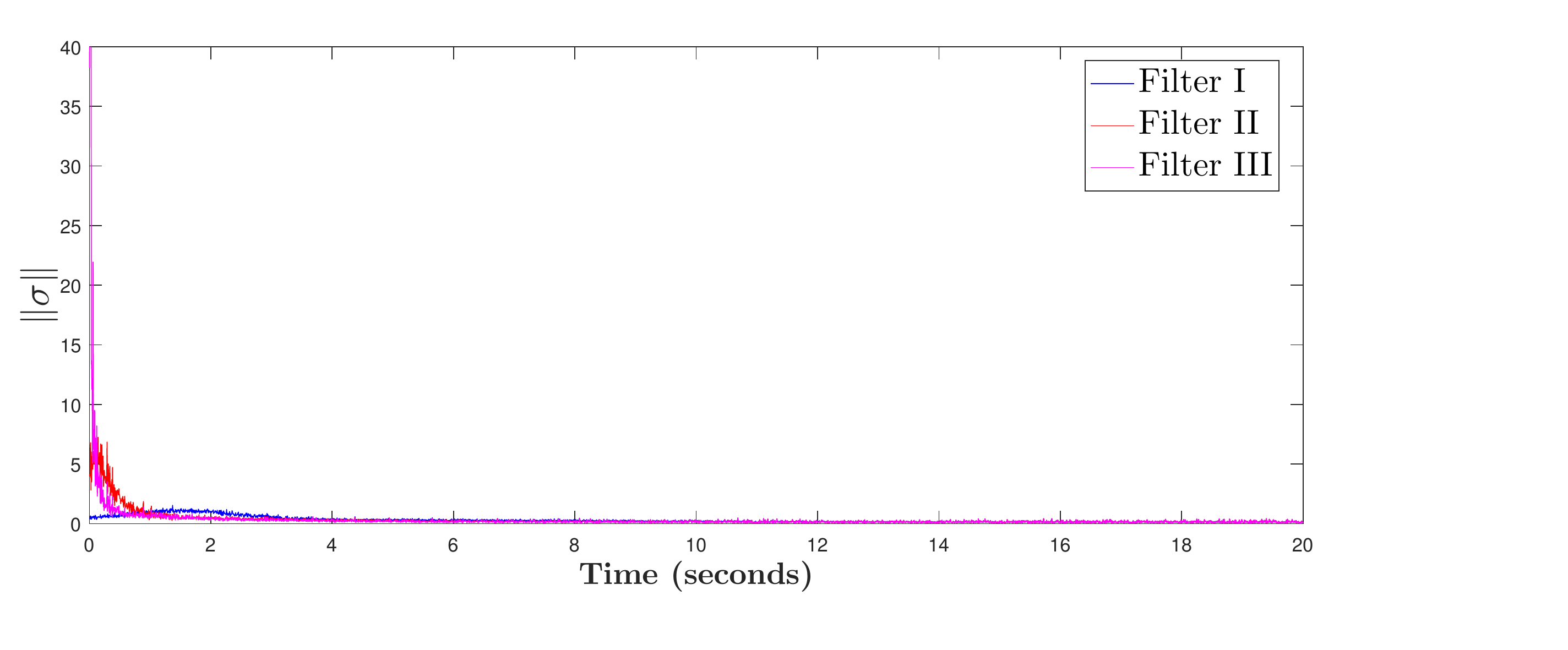}
\caption{Innovation term $\sigma$ versus time.}
\label{figure::sigma}
\end{figure}
The innovation term $\sigma$ for the three filters is plotted in Fig. \ref{figure::sigma}. It can be seen that Filter III innovation term is very aggressive at initial times compared to the two other filters. Although this allows fast correction of the attitude. Filter II has a relatively less aggressive (compared to Filter III) correction while maintaining a considerably good speed (compared to Filter I).

\section{Conclusion}
The traditional nonlinear complementary filter on $SO(3)$ has been revisited and an explicit time-solution of the resulting attitude estimation error has been provided in the bias-free case. Almost global asymptotic (and local exponential) stability properties of this filter, that are usually determined using complex proofs, are easily deduced from the obtained closed form solution. The robustness of this filter has also been investigated and it has been shown that this filter is not ISS with respect to gyro measurement disturbances. As an alternative solution, we consider two nonlinear complementary filters (with state-dependent gains) and provide explicit solutions for the resulting estimation error. It is shown that these proposed observers lead to better results in terms of convergence and robustness to measurement errors. The domain of Local ISS (with respect to gyro errors) for the three discussed nonlinear complementary filters on $SO(3)$ is explicitly computed and reveals that the state dependent gain filters have a larger robustness domain with respect to angular velocity measurement errors. On the other hand, the robustness domain for the three filters with respect to small attitude measurement errors is shown to be the same for the three filters. 
\appendices
\section{Proof of Lemma \ref{lemma::1}}\label{appendix::proof::lemma1}
Let $R\in SO(3)$ be an attitude matrix represented by a rotation of angle $\theta$ around some unit vector $u\in\mathbb{S}^2$. Using \eqref{Rod_formula} and the fact that $[u]_\times^2=-u^\top uI+uu^\top$, one can show that
\begin{align*}
|R|_I^2=\frac{1}{4}\mathrm{tr}(I-R)=\frac{1}{2}(1-\cos(\theta))=\sin^2(\theta/2).
\end{align*}
On the other hand, the rotation matrix $R^2$ represents a rotation of the same direction $u$ as $R$ and with twice the angle of rotation of $R$. Hence, one has $|R^2|_I^2=\sin^2(\theta)=4\cos^2(\theta/2)\sin^2(\theta/2)=4(1-|R|_I^2)|R|_I^2$. Moreover, one has
\begin{align}\nonumber
\|\psi(R)\|^2&=-\frac{1}{2}\mathrm{tr}(\mathbb{P}_a(R),\mathbb{P}_a(R))=\mathrm{tr}(I-R^2)/4=|R^2|_I^2,
\end{align}
for all $R\in SO(3)$. On the other hand, using the fact that $\mathrm{tr}(A[u]_\times)=0$ for any symmetric matrix $A$ and $u\in\mathbb{R}^3$ and $[u]_\times^2=-u^\top uI+uu^\top$, one obtains
\begin{align*}
\mathrm{tr}(A(I-R))&=-\mathrm{tr}(A(\sin(\theta)[u]_\times+(1-\cos(\theta))[u]_\times^2)),\\
							  &=-(1-\cos(\theta))\mathrm{tr}(A[u]_\times^2)),\\
							  &=(1-\cos(\theta))u^\top\bar Au,\\
							  &=2|R|_I^2u^\top\bar Au.
\end{align*}
Moreover, one has $\lambda_{\max}^{\bar A}\|u\|^2\leq u^\top\bar Au\leq\lambda_{\max}^{\bar A}\|u\|^2$ and $\|u\|^2=1$ which proves \eqref{bounds::VA}.

Let $R\in SO(3)\setminus\Pi$ and hence $R=\mathcal{R}_r(\mathcal{Z}(R))$. In view of \eqref{Cayley_formula} one has
\begin{align*}
\mathbb{P}_a(AR)&=\frac{1}{2}(AR-R^\top A)\\
							&=\frac{1}{1+\|\mathcal{Z}(R)\|^2}\big(A\mathcal{Z}(R)\mathcal{Z}(R)^\top-\mathcal{Z}(R)\mathcal{Z}(R)^\top A\\
							&\hspace{3.5cm}+A[\mathcal{Z}(R)]_\times+[\mathcal{Z}(R)]_\times A\big)\\
							&=\frac{[\mathcal{Z}(R)\times A\mathcal{Z}(R)]_\times+[\bar A\mathcal{Z}(\tilde R)]_\times}{1+\|\mathcal{Z}(R)\|^2},
\end{align*}
where equalities $yx^\top-xy^\top=[x\times y]_\times$ and $M^\top[x]_\times+[x]_\times M+[Mx]_\times=\mathrm{tr}(M)[x]_\times$, for all $x,y\in\mathbb{R}^3$ and $M\in\mathbb{R}^{3\times 3}$, have been used. Consequently, one obtains
\begin{align*}
\psi(AR)&=\frac{\mathcal{Z}(R)\times A\mathcal{Z}(R)+\bar A\mathcal{Z}(\tilde R)}{1+\|\mathcal{Z}(R)\|^2}\\
			  &=\frac{(I-[\mathcal{Z}(R)]_\times)}{1+\|\mathcal{Z}(R)\|^2}\bar A\mathcal{Z}(R).
\end{align*}
It follows that
\begin{align*}
\|\psi(AR)\|^2&=\frac{\mathcal{Z}(R)^\top\bar A(I+[\mathcal{Z}(R)]_\times)(I-[\mathcal{Z}(R)]_\times)\bar A\mathcal{Z}(R)}{(1+\|\mathcal{Z}(R)\|^2)^2}\\
					 &=\frac{\mathcal{Z}(R)^\top\bar A(I-[\mathcal{Z}(R)]_\times^2)\bar A\mathcal{Z}(R)}{(1+\|\mathcal{Z}(R)\|^2)^2}\\
					 &=\frac{\mathcal{Z}(R)^\top\bar A^2\mathcal{Z}(R)}{1+\|\mathcal{Z}(R)\|^2}-\frac{\big(\mathcal{Z}(R)^\top\bar A\mathcal{Z}(R)\big)^2}{\big(1+\|\mathcal{Z}(R)\|^2\big)^2}\\
					 &=\frac{\|\bar A\mathcal{Z}(R)\|^2}{1+\|\mathcal{Z}(R)\|^2}\left(1-\frac{\|\mathcal{Z}(R)\|^2\cos^2(\phi)}{1+\|\mathcal{Z}(R)\|^2}\right),
\end{align*}
where $\phi=\angle(\mathcal{Z}(R),\bar A\mathcal{Z}(R))$. On the other hand, one has
\begin{multline*}
\lambda_{\min}^{\bar A}\|\mathcal{Z}(R)\|^2\leq\mathcal{Z}(R)^\top\bar A\mathcal{Z}(R)=\|\mathcal{Z}(R)\|\|\bar A\mathcal{Z}(R)\|\cos(\phi)\\\leq\|\mathcal{Z}(R)\|^2\|\bar A\|_2\cos(\phi)=\lambda_{\max}^{\bar A}\|\mathcal{Z}(R)\|^2\cos(\phi),
\end{multline*}
which implies that
$$
\xi=\frac{\lambda_{\min}^{\bar A}}{\lambda_{\max}^{\bar A}}\leq\cos(\phi)\leq 1.
$$
Consequently, it follows that
\begin{align*}
\|\psi(AR)\|&\leq\big(\lambda_{\max}^{\bar A}\big)^2\frac{\|\mathcal{Z}(R)\|^2}{1+\|\mathcal{Z}(R)\|^2}\left(1-\frac{\|\mathcal{Z}(R)\|^2\xi^2}{1+\|\mathcal{Z}(R)\|^2}\right)\\
				 &\geq\big(\lambda_{\min}^{\bar A}\big)^2\frac{\|\mathcal{Z}(R)\|^2}{1+\|\mathcal{Z}(R)\|^2}\left(1-\frac{\|\mathcal{Z}(R)\|^2}{1+\|\mathcal{Z}(R)\|^2}\right),
\end{align*}
which implies identity \eqref{psiA_ineq} in view of the fact that
\begin{align}\label{norm::Z}
\|\mathcal{Z}(R)\|^2=\frac{\|\psi(R)\|^2}{4(1-|R|_I^2)^2}=\frac{|R|_I^2}{1-|R|_I^2}.
\end{align}

\bibliographystyle{IEEETran}
\bibliography{IEEEabrv,Hybrid}
\begin{IEEEbiography}[{\includegraphics[width=1in,height=1.25in,clip,keepaspectratio]{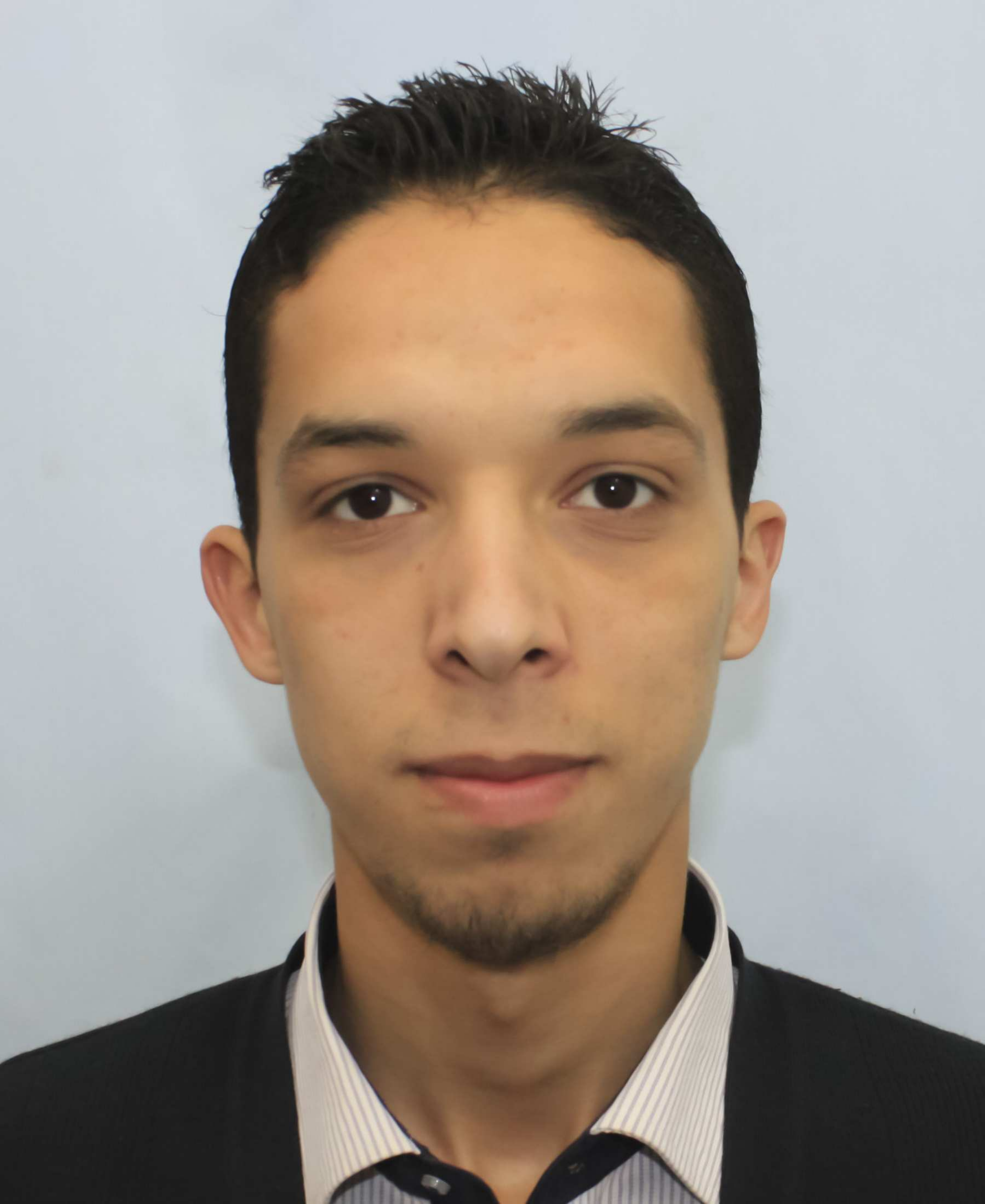}}]{Soulaimane Berkane}
received his B.Sc. and M. Sc. degrees in Automatic Control from Ecole Nationale Polytechnique, Algiers, in 2013. He is currently a Ph. D. candidate and a Research Assistant at the department of Electrical and Computer Engineering at the University of Western Ontario, Canada. His research interest focuses on nonlinear and hybrid control with application to geometric attitude control and estimation.
\end{IEEEbiography}

\begin{IEEEbiography}[{\includegraphics[width=1in,height=1.25in,clip,keepaspectratio]{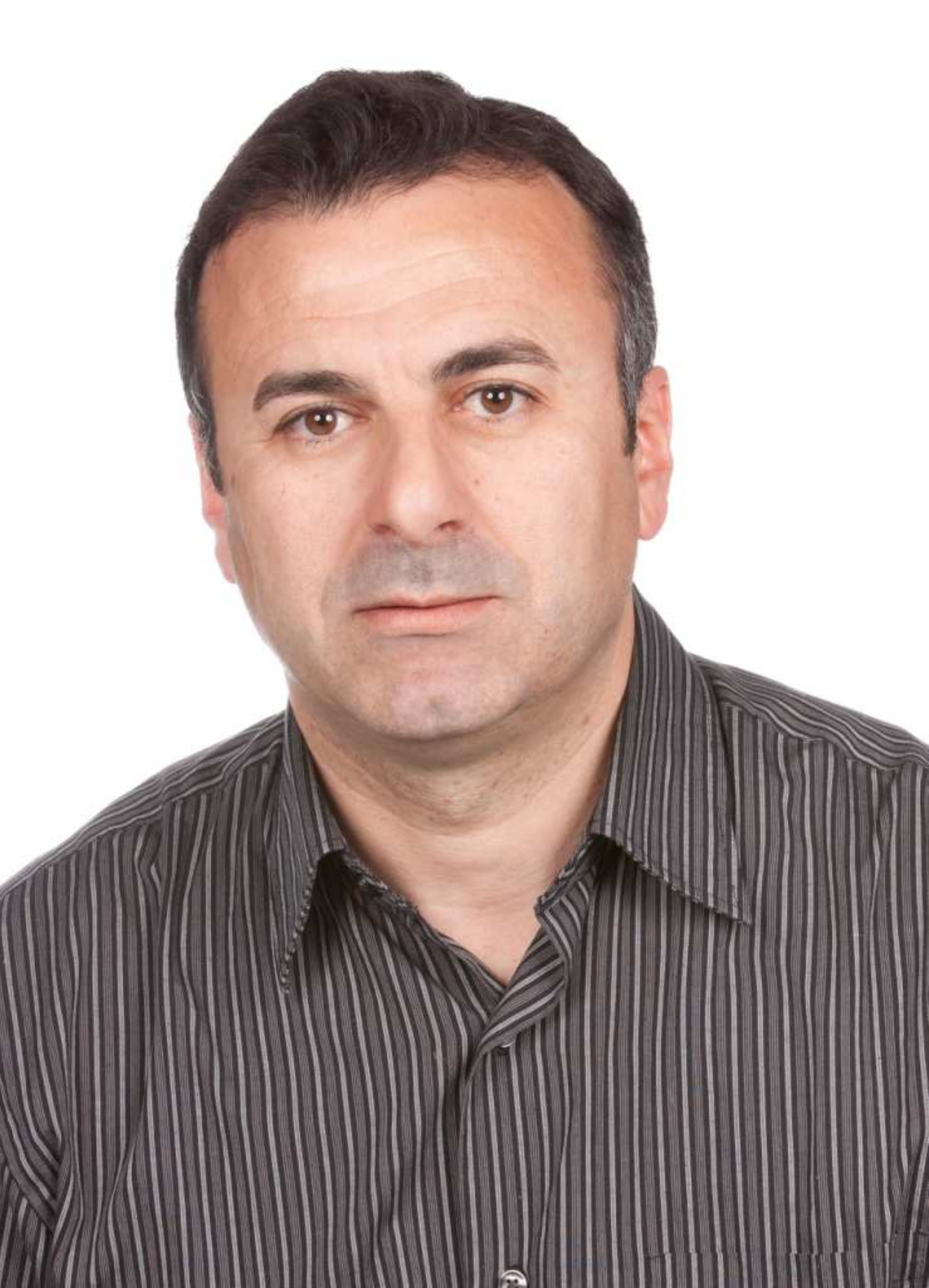}}]{Abdelhamid Tayebi}
received his B. Sc. in Electrical Engineering from Ecole Nationale Polytechnique, Algiers, in 1992, his M. Sc. in robotics from Universit\'e Pierre \& Marie Curie, Paris, France in 1993, and his Ph. D. in Robotics and Automatic Control from Universit\'e de Picardie Jules Verne, France in December 1997. He joined the department of Electrical Engineering at Lakehead University in December 1999 where he is presently a full Professor. He is a Senior Member of IEEE and serves as an Associate Editor for Automatica, IEEE Transactions on Control Systems Technology, Control Engineering Practice and IEEE CSS Conference Editorial Board. He also served as an Associate Editor for IEEE Transactions on Cybernetics (2006-2014). He is a member of the board of Directors of IFAC Canada. He is the founder and Director of the Automatic Control Laboratory at Lakehead University.  His research interests are related to Control Engineering in general with applications to unmanned aerial vehicles.
\end{IEEEbiography}

\end{document}